\documentclass{article}
\usepackage[utf8]{inputenc}

\title{The over-topos at a model}
\author{Olivia Caramello \and  Axel Osmond}
\date{April 12, 2021}

\usepackage{amsmath}

\usepackage{comment}
\usepackage{leftidx}
\usepackage{multicol}
\usepackage{amssymb}
\usepackage{stmaryrd}
\usepackage{accents}

\usepackage{xfrac}
\usepackage{amsthm}
\usepackage{xcolor}
\usepackage{sectsty}
\usepackage{relsize}

\usepackage{indentfirst}
\usepackage{tikz}
\tikzset{shorten <>/.style={shorten >=#1,shorten <=#1}}

\usepackage{hyperref}
\hypersetup{
    colorlinks=true,
    linkcolor=blue,
    filecolor=magenta,      
    urlcolor=cyan,
}
\usepackage{cleveref}

\usetikzlibrary{matrix,arrows,decorations.pathmorphing}
\usepackage{tikz-cd}
\newcounter{nodemaker}
\tikzset{%
    symbol/.style={%
        draw=none,
        every to/.append style={%
            edge node={node [sloped, allow upside down, auto=false]{$#1$}}}
    }
}

\makeatletter
\newcommand{\bigdoublevee}{\big@doubleop{\bigvee}}
\newcommand{\bigdoublewedge}{\big@doubleop{\bigwedge}}
\newcommand{\big@doubleop}[1]{%
  \DOTSB\mathop{\mathpalette\big@doubleop@aux{#1}}\slimits@
}

\newcommand\big@doubleop@aux[2]{%
  \sbox\z@{$\m@th#1#2$}%
  \makebox[1.35\wd\z@][s]{$\m@th#1#2\hss#2$}%
}
\makeatother

\makeatletter
\newcommand*{\doublerightarrow}[2]{\mathrel{
  \settowidth{\@tempdima}{$\scriptstyle#1$}
  \settowidth{\@tempdimb}{$\scriptstyle#2$}
  \ifdim\@tempdimb>\@tempdima \@tempdima=\@tempdimb\fi
  \mathop{\vcenter{
    \offinterlineskip\ialign{\hbox to\dimexpr\@tempdima+1em{##}\cr
    \rightarrowfill\cr\noalign{\kern.5ex}
    \rightarrowfill\cr}}}\limits^{\!#1}_{\!#2}}}
\newcommand*{\triplerightarrow}[1]{\mathrel{
  \settowidth{\@tempdima}{$\scriptstyle#1$}
  \mathop{\vcenter{
    \offinterlineskip\ialign{\hbox to\dimexpr\@tempdima+1em{##}\cr
    \rightarrowfill\cr\noalign{\kern.5ex}
    \rightarrowfill\cr\noalign{\kern.5ex}
    \rightarrowfill\cr}}}\limits^{\!#1}}}
\makeatother

\usepackage{geometry}
\newtheorem{theorem}{Theorem}[section]
\theoremstyle{proposition}
\newtheorem{proposition}[theorem]{Proposition}
\newtheorem{corollary}[theorem]{Corollary}
\newtheorem{corollary'}[theorem]{Corollary}
\newtheorem{lemma}[theorem]{Lemma}
\theoremstyle{definition}
\newtheorem{definition}[theorem]{Definition}

\theoremstyle{remark}
\newtheorem*{remark}{Remark}

\theoremstyle{remarks}
\newtheorem*{remarks}{Remarks}

 \newcommand{\Lex}
 {{\bf Lex}}

 \newcommand{\GTop}
 {{\bf GTop}}
 
 \newcommand{\Geom}
 {{\bf Geom}}
 
 \newcommand{\Cat}
 {{\bf Cat}}

\newcommand{\cod}
 {{\rm cod}}

\newcommand{\ac}  
{`}

\newcommand{\comp}
 {\circ}

\newcommand{\Cont}
 {{\bf Cont}}

\newcommand{\dom}
 {{\rm dom}}
 
\newcommand{\fp}
{{\rm fp }}

\newcommand{\y}
{{\rm y }}

\newcommand{\li}
{{\textup{lim } }}

\newcommand{\oplaxlim}
{{\textup{oplaxlim } }}

\newcommand{\oplaxcolim}
{{\textup{oplaxcolim } }}

\newcommand{\lan}
{{\textup{lan } }}

\newcommand{\ran}
{{\textup{ran } }}

\newcommand{\colim}
{{\textup{colim}}}

\newcommand{\comma}[2]
{\mbox{$(#1\!\downarrow\!#2)$}}

\newcommand{\empstg}
 {[\,]}

\newcommand{\epi}
 {\twoheadrightarrow}

\newcommand{\hy}
 {\mbox{-}}

\newcommand{\im}
 {{\rm Im}}

\newcommand{\imp}
 {\!\Rightarrow\!}

\newcommand{\Ind}
 {{\rm Ind}}
 
 \newcommand{\Pro}
 {{\rm Pro}}

\newcommand{\mono}
 {\rightarrowtail}

\newcommand{\ob}
 {{\rm ob}}
 
 \newcommand{\Hom}
 {{\rm Hom}}

\newcommand{\op}
 {^{\rm op}}
 
 \newcommand{\pt}
 {{\bf pt}}

\newcommand{\Set}
 {{\bf Set }}

\newcommand{\Sh}
 {{\bf Sh}}
 
 \newcommand{\St}
 {{\bf St}}

\newcommand{\Desc}
 {{\rm Desc}}

\newcommand{\Sub}
 {{\rm Sub}}

\sectionfont{\large}
\subsectionfont{\normalsize}

\begin{document}
\maketitle

\begin{abstract}
    With a model of a geometric theory in an arbitrary topos, we associate a site obtained by endowing a category of generalized elements of the model with a Grothendieck topology, which we call the antecedent topology. Then we show that the associated sheaf topos, which we call the over-topos at the given model, admits a canonical totally connected morphism to the given base topos and satisfies a universal property generalizing that of the colocalization of a topos at a point. We first treat the case of the base topos of sets, where global elements are sufficient to describe our site of definition; in this context, we also introduce a geometric theory classified by the over-topos, whose models can be identified with the model homomorphisms towards the (internalizations of the) model. Then we formulate and prove the general statement over an arbitrary topos, which involves the stack of generalized elements of the model. Lastly, we investigate the geometric and $2$-categorical aspects of the over-topos construction, exhibiting it as a bilimit in the bicategory of Grothendieck toposes.
\end{abstract}

\section*{Notation}

The notation employed in the paper will be standard; in particular,

\begin{itemize}
    \item We shall denote by $\Set$ the category of sets (within a fixed model of set theory).
    
    \item Given a geometric theory $\mathbb T$, we shall denote by $({\cal C}_{\mathbb T}, J_{\mathbb T})$ its geometric syntactic site and by   $\Set[{\mathbb T}]$ its classifying topos, which, as is well-known, can be represented as $\Sh({\cal C}_{\mathbb T}, J_{\mathbb T})$ (for background on classifying toposes, the reader may refer to \cite{OCBook}).
    
    \item For any geometric theory $\mathbb T$, we denote by ${\mathbb T}[{\cal E}]$ the category of $\mathbb T$-models in a geometric category $\cal E$ and model homomorphisms between them. For any $\cal E$, we have an equivalence ${\mathbb T}[{\cal E}]\simeq \textup{\bf Cart}_{J_{\mathbb T}}({\cal C}_{\mathbb T}, {\cal E})$ between ${\mathbb T}[{\cal E}]$ and the category of cartesian (that is, finite-limit-preserving) functors ${\cal C}_{\mathbb T}\to {\cal E}$ which are $J_{\mathbb T}$-continuous (that is, which send $J_{\mathbb T}$-covering families to covering families in $\cal E$). The functor ${\cal C}_{\mathbb T}\to {\cal E}$ corresponding to a $\mathbb T$-model $M$ in $\cal E$ will be denoted by $F_{M}$; it sends any geometric formula $\{\vec{x}. \phi\}$ over the signature of $\mathbb T$ to its interpretation $[[\vec{x}. \phi]]_{M}$ in $M$, and acts accordingly on arrows (for more details, see, for instance, Chapter 1 of \cite{OCBook}).
\end{itemize}

\section*{Introduction}

The main topic of this paper is the construction of a topos classifying model homomorphisms towards (the internalizations of) a fixed model of a geometric theory in a given Grothendieck topos. More precisely, the desired universal property can be formulated as follows. Let $ \mathbb{T}$ a geometric theory over a signature $\mathcal{L}$, whose geometric syntactic site will be denoted $ (\mathcal{C}_\mathbb{T}, J_\mathbb{T})$, $\mathcal{E}$ a Grothendieck topos and $M$ a $\mathbb{T}$-model in $\mathcal{E}$. We want to construct the \emph{$\mathbb{T}$-over-topos} associated to $M$, that is, a geometric morphism $ u_{M} : \mathcal{E}[M] \rightarrow \mathcal{E} $ satisfying the universal property that for any $ \mathcal{E}$-topos $ g : \mathcal{G} \rightarrow \mathcal{E} $ one has an equivalence of categories
\[ \textup{\bf Geom}_\mathcal{E} [g, u_{M}] \simeq \mathbb{T}[\mathcal{G}]/g^*(M). \]
In words, we want $ \mathcal{E}[M]$ to classify the theory of $ \mathcal{L}$-structures homomorphisms from a $\mathbb{T} $-model to (internalizations of) $M$. In some sense, this is a way of forcing $M$ to become terminal amongst $ \mathbb{T}$-models. In this paper, we explicitly construct a site of definition for the over-topos from model-theoretic data, and describe both the logical and geometrical aspects of this construction.\\

In section $1$, we restrict to the case of a set-based model, where the construction of our site of definition for the over-topos simplifies thanks to specific properties of the terminal object of the topos $\Set$ of sets.\\

In section $2$, in preparation for the generalization of our construction to an arbitrary topos, we introduce a number of stacks that will be used in the sequel. Of particular relevance for our purposes is the notion of \emph{lifted topology} on a category of the form $(1_{\cal F}\downarrow f^{\ast})$, where $f:{\cal F}\to {\cal E}$ is a geometric morphism, as the smallest topology which makes both projection functors to $\cal E$ and $\cal F$ comorphisms to the canonical sites on $\cal E$ and $\cal F$; we provide a fully explicit description of this topology by providing a basis for it.\\

In section $3$, we generalize the construction of the over-topos to a model in an arbitrary Grothendieck topos. For this, we construct a canonical stack associated with this model and apply Giraud's general construction of the classifying topos of a stack to prove the desired universal property. \\ 

In section $4$ we investigate the over-topos construction from a bicategorical point of view, relating it with classical constructions in the bicategory of Grothendieck toposes, as well as with the notion of totally connected topos. In the particular case of a set-based model, that is, of a point $f_M: \Set \rightarrow \Set[\mathbb{T}]$ of the classifying topos for $\mathbb T$, the $\mathbb{T}$-over-topos at $M$ has as category of points the comma category $(\mathbb{T}[\Set]\downarrow M)$ and is shown to coincide with the colocalization of $\Set[\mathbb{T}]$ at $f_{M}$. This construction is dual, in a sense, to the notion of \emph{Grothendieck-Verdier localization} at a point, which classifies morphisms whose \emph{domain} is a fixed model. \\

\section{Over-topos of a set-based $\mathbb{T}$-model}

We suppose in this section that $ \mathcal{E}$ is equal to the topos $\Set$ of sets (within a fixed model of set theory). We first list some specific properties of the terminal object $1$ of $ \Set$, we are going to make use of in this section, and which may fail in arbitrary toposes, as we are going to see in section 3 which will addres s the general case. First, $ \Set$ is generated under coproducts by $ 1$, that is, for any set $ X $ one has
\[ X \simeq \coprod_{\textup{Hom}_{\Set}(1,X)} 1  \]
Moreover, 1 is \emph{projective}, that is, any epimorphism $ X \twoheadrightarrow 1$ admits a section $ 1 \rightarrow X$, and \emph{indecomposable}, which means that for any arrow from $1$ to a coproduct $ \coprod_{i \in I} X_i$, there is a section for at least one $i \in I$:
\[ 
\begin{tikzcd}
X_i \arrow[r, tail]                       & \underset{i \in I}{\coprod} X_i  \\
1 \arrow[r, equal] \arrow[u, dashed] & 1 \arrow[u]               
\end{tikzcd} \]

We are going to use these properties in order to define a cartesian, subcanonical site for the $ \mathbb{T}$-over-topos associated with $M$. This will involve the \emph{category of elements} of $M$, and a certain topology related to the syntactic topology $J_{\mathbb T}$ of $ \mathbb{T}$. Let $ \mathcal{C}_\mathbb{T}$ denote the geometric syntactic of $\mathbb{T}$, and, for any object $ \{\vec{x}^{\vec{A}}. \phi\}$ of $\mathcal{C}_\mathbb{T}$, that is, a geometric formula $ \phi $ in the signature $ \mathcal{L}$ in the context $\vec{x}$ of sort $\vec{A}$, $ \llbracket \vec{x}^{\vec{A}}.\phi \rrbracket_M$ the interpretation of $\{\vec{x}^{\vec{A}}. \phi\}$ in $ M $. In particular the category $ \int M$ of elements of $M$ - seen as a geometric functor $F_{M}:{\cal C}_{\mathbb T}\to \Set$ - has as objects the pairs $(\{ \vec{x}^{\vec{A}}.\phi \}, \vec{a})$, where $ \vec{a} \in \llbracket \vec{x}^{\vec{A}}.\phi\rrbracket_M $ is a global element of $M\vec{A}$ satisfying the formula $\phi$, and as arrows $(\{ \vec{x}_1^{\vec{a_1}}.\phi_1 \}, \vec{a_1}) \rightarrow (\{ \vec{x}_2^{\vec{a_2}}.\phi_2 \}, \vec{a_2}) $ the arrows $ [\theta] : \{ \vec{x}_1^{\vec{a_1}}.\phi_1 \} \rightarrow \{ \vec{x}_2^{\vec{a_2}}.\phi_2 \}_1$ in ${\cal C}_{\mathbb T}$ such that $\llbracket \theta \rrbracket_M (\vec{a_1}) = \vec{a_2}$, that is diagramatically in $ \Set$:
\[ \begin{tikzcd}
& 1 \arrow[]{ld}[swap]{ \vec{a_1}  }  \arrow[]{rd}{ \vec{a_2} } & \\
\llbracket \vec{x}_1^{A_1}.\phi_1\rrbracket_M \arrow[]{rr}{\llbracket\theta \rrbracket_M} & & \llbracket \vec{x}_2^{A_2}.\phi_2\rrbracket_M 
\end{tikzcd}\]

We shall find it convenient to present our Grothendieck topologies in terms of bases generating them. Recall that a basis $\cal B$ for a Grothendieck topology on a category $\cal C$ is a collection of \emph{presieves} on objects of $\cal C$ (by a presieve we simply mean a small family of arrows with common codomain) satisfying the following properties (where we denote by ${\cal B}(c)$, for an object $c$ of $\cal C$, the collection of presieves in $\cal B$ on the object $c$):

\begin{enumerate}
		\item[(a)] If $f$ is the identity then $\{f\}$ lies in ${\cal B}(\cod(f))$.
		
		\item[(b)] If $R\in {\cal B}(c)$ then for any arrow $g:d\rightarrow c$ in $\cal C$ there exists a presieve $T$ in ${\cal B}(d)$ such that for each $t\in T$, $g\circ t$ factors through some arrow in $R$.
		
		\item[(c)] $\cal B$ is closed under ``multicomposition'' of families; that is, given a presieve $\{f_{i}:c_{i}\rightarrow c \mid i\in I \}$ in ${\cal B}(c)$ and for each $i\in I$ a presieve $\{g_{ij}:d_{ij}\rightarrow c_{i} \mid j\in I_{i} \}$ in ${\cal B}(c_{i})$, the ``multicomposite'' presieve $\{\,f_{i}\circ g_{ij}:d_{ij}\rightarrow c \mid i\in I, j\in I_{i} \}$ belongs to ${\cal B}(c)$.
\end{enumerate}

The Grothendieck topology generated by a basis has as covering sieves precisely those which contain a presieve in the basis.

As it is well known, the Grothendieck topology $J_{\mathbb T}$ on ${\cal C}_{\mathbb T}$ has as a basis the collection ${\cal B}_{\mathbb T}$ of small families $\{ [\theta_i]:\{\vec{x_i}. \phi_i\} \to \{\vec{x}. \phi\} \mid i\in I\}$ such that the sequent 
\[
(\phi \vdash_{\vec{x}} \bigvee_{i\in I} (\exists \vec{x_{i}})\theta_i(\vec{x_{i}}, \vec{x}))
\]
is provable in $\mathbb T$.  

The fact that the following definition is well-posed in ensured by the subsequent Lemma.

\begin{definition}\label{antecedents topology}
Let $M$ be a set-based model of a geometric theory $\mathbb T$. We define the \emph{antecedents topology at $M$} as the Grothendieck topology $J^{\textup{ant}}_M$ on $ \int M$ generated by the basis ${\cal B}_{\mathbb T}^{\textup{ant}}$ consisting of the families 
\[  \big{(}(\vec{b},\{ \vec{x}_i^{\vec{A}_i}. \phi_i \}) \stackrel{\lbrack \theta_i \rbrack }{\longrightarrow} (\vec{a}, \{ \vec{x}^{\vec{A}}. \phi \})\big{)}_{i\in I, \,\vec{b} \textup{ }\mid \textup{ } \llbracket \theta_i \rrbracket_M(\vec{b})=\vec{a} }  
\]
(indexed by the objects $(\vec{a}, \{\vec{x}. \phi\})$ of $\int M$ and the families $( \lbrack \theta_i \rbrack)_{i \in I} $ in ${\cal B}_{\mathbb T}$ on $\{\vec{x}. \phi\}$) consisting of all the ``antecedents'' $\vec{b}$ of a given $\vec{a} \in \llbracket \vec{x}^{\vec{A}}. \phi \rrbracket_M $ with respect to some $\lbrack \theta_i \rbrack$. 
\end{definition}

\begin{lemma}
The collection ${\cal B}_{\mathbb T}^{\textup{ant}}$ of sieves in $ \int M$ is a basis for a Grothendieck topology.
\end{lemma}
\begin{proof}
\begin{itemize}
    \item Condition (a) is trivially satisfied. 
    
    \item Condition (b): given an arrow $\lbrack \theta \rbrack : (\vec{c},\{ y^{\vec{B}}. \psi \}) \rightarrow  (\vec{a}, \{ \vec{x}^{\vec{A}}. \phi \})$ in $ \int M$ and a family $ \lbrack \theta_i \rbrack_\mathbb{T } : \{ \vec{x}_i^{\vec{A}_i}. \phi_i \} \rightarrow \{ \vec{x}^{\vec{A}}. \phi \})_{i \in I} $ in ${\cal B}_{\mathbb T}$ on $\{ \vec{x}^{\vec{A}}. \phi \}$, we have the following pullback squares in $\int M$ for each antecedent $ \vec{b}$ of $\vec{a}$: 
    \[ \begin{tikzcd}
    ((\vec{c},\vec{b}),\{ \vec{y}^{\vec{B}},\vec{x}_i^{\vec{A}_i}. \theta(\vec{y})=\theta_i(\vec{x}_i) \})  \arrow[phantom, very near start]{rd}{\lrcorner} \arrow[]{r}{} \arrow[]{d}[swap]{\lbrack \theta^*\theta_i \rbrack} & (\vec{b},\{ \vec{x}_i^{\vec{A}_i}. \phi_i \}) \arrow[]{d}{\lbrack \theta_i \rbrack}) \\  (\vec{c},\{ \vec{y}^{\vec{B}}. \psi \}) \arrow[]{r}{\lbrack \theta \rbrack} & (\vec{a}, \{  \vec{x}^{\vec{A}}. \phi  \})
    \end{tikzcd}\]
    
    Note that the family $\{[\theta^*\theta_i] \mid i\in I\}$ lies in ${\cal B}_{\mathbb T}(\{\vec{y}^{\vec{B}}. \psi\})$ (as ${\mathcal B}_{\mathbb T}$ is stable under pullback). So the family
	\[
	\big{(}[\theta^*\theta_i]:((\vec{c},\vec{b}),\{ \vec{y}^{\vec{B}},\vec{x}_i^{\vec{A}_i}. \theta(\vec{y})=\theta_i(\vec{x}_i) \}) \to (\vec{c},\{ \vec{y}^{\vec{B}}. \psi \})\big{)}_{i\in I} 
	\] 
	is the family of antecedents of $\vec{c}$ indexed by it, whence it lies in ${\mathcal B}_{\mathbb T}^{\textup{ant}}$, as desired; indeed,
	$ \llbracket \theta^*\theta_i \rrbracket_M = \llbracket \theta \rrbracket_M ^* \llbracket \theta_i \rrbracket_M $ since $F_{M}$ is cartesian.
	
    \item Condition (c) follows immediately from the fact that ${\mathcal B}_{\mathbb T}$ is a basis for $J_{\mathbb T}$. 
\end{itemize}\end{proof}

\begin{remark}
For a family $ ( \lbrack \theta_i \rbrack)_{i \in I} $ in ${\cal B}_{\mathbb T}(\{\vec{x}. \phi\})$ and a global element $ \vec{a}$ of the interpretation $\llbracket \vec{x}^{\vec{A}}. \phi \rrbracket_M$ of its codomain, the fiber at $\vec{a}$ of some $\llbracket\theta_i\rrbracket_M$ can be categorically characterized as the following pullback:
\[ 
\begin{tikzcd}[column sep=large]
 \llbracket \theta_i \rrbracket_M ^{-1}(\vec{a}) \arrow[d] \arrow[r] \arrow[rd, "\lrcorner", very near start, phantom]         & 1 \arrow[d, "\vec{a}"]                                \\
 \llbracket \vec{x}_i^{\vec{A}_i}. \phi_i \rrbracket_M \arrow[r, " \llbracket \theta_i \rrbracket_M"] & \llbracket \vec{x}^{\vec{A}}. \phi \rrbracket_M
\end{tikzcd} \]
As $F_{M}$ is a $J_\mathbb{T}$-continuous cartesian functor, it sends $J_\mathbb{T}$-covering families to jointly surjective families in $\Set$, so by the stability of epimorphisms under pullback the global fiber of the cover at $\vec{a}$ is also an epimorphism:
\[ 
\begin{tikzcd}
\langle \llbracket \theta_i \rrbracket_M \rangle_{i \in I}^{-1}(\vec{a}) \arrow[d] \arrow[rr, two heads] \arrow[rrd, "\lrcorner", very near start, phantom] &        & 1 \arrow[d, "\vec{a}"]                                \\
\underset{i \in I}{\coprod} \llbracket \vec{x}_i^{\vec{A}_i}. \phi_i \rrbracket_M \arrow[rr, "\langle \llbracket \theta_i \rrbracket_M \rangle_{i \in I}", two heads] & & \llbracket \vec{x}^{\vec{A}}. \phi \rrbracket_M
\end{tikzcd} \]
Note that, by the stability of colimits under pullback, the global fiber of $\vec{a}$ decomposes as the coproduct 
\[ \langle \llbracket \theta_i \rrbracket_M \rangle_{i \in I}^{-1}(\vec{a}) \simeq \coprod_{i \in I}\llbracket \theta_i \rrbracket_M ^{-1}(\vec{a}). \]

One can then identify the set of antecedents of $ \vec{a}$ with the set of global elements $ \vec{b} : 1 \rightarrow  \langle \llbracket \theta_i \rrbracket_M \rangle_{i \in I}^{-1}(\vec{a})$, which decomposes as the disjoint union of the sets of global elements of the fibers $ \llbracket \theta_i \rrbracket_M ^{-1}(\vec{a})$ (since $1$ is indecomposable and coproducts are disjoint and stable under pullback). 

Note however that the projection $ \int M \rightarrow \mathcal{C}_\mathbb{T}$ does \emph{not} send $J^{\textup{ant}}_{M}$-covers to $ J_\mathbb{T}$-covers as $\vec{a} \in \llbracket \vec{x}^{\vec{A}}. \phi \rrbracket_M$ may have antecedents in only some of the $ \llbracket \vec{x}_i^{\vec{A}_i}. \phi_i \rrbracket_M$.
\end{remark}

\begin{proposition}
The category $ \int M $ of elements of $M$ is geometric.
\end{proposition}

\begin{proof}
Actually all the properties we have to check are inherited from the geometricity of $\mathcal{C}_\mathbb{T}$ and the fact that $M$ is a model of $\mathbb{T}$: \begin{itemize}
    \item $\int M$ is cartesian: for any finite diagram $ D \rightarrow \int M$, the underlying diagram in $\mathcal{C}_\mathbb{T}$ has a limit \[ \underset{d \in D}{\lim} \{ \vec{x}_{d}^{\vec{A}_{d}}. \phi_{d} \} = \{ (\vec{x}_d^{\vec{A}_d})_{d \in D}, \bigwedge_{\delta : d \rightarrow d'} \theta_{\delta}(\vec{a}_d)=\vec{a}_{d'} \} \] which is sent to a limit in $\Set$ by the cartesian functor $F_{M}$; note that an element of $ \underset{d \in D}{\lim} \; \llbracket \vec{x}_{d}^{\vec{A}_{d}}. \phi_{d} \rrbracket_M  $ is a family $ (\vec{a}_d)_{d \in D} $ with $ \vec{a}_d \in \llbracket \vec{x}_{d}^{\vec{A}_{d}}. \phi_{d} \rrbracket_M$ and $ \llbracket \theta_\delta \rrbracket_M(\vec{a}_d) = \vec{a}_{d'}$ for each transition morphism $ \delta : d \rightarrow d'$ in $D$. This exactly says that 
    \[ ((\vec{a}_d)_{d \in D}, \underset{d \in D}{\lim} \; \{ \vec{x}_{d}^{\vec{A}_{d}}. \phi_{d} \}) = \underset{d \in D}{\lim} \; (a_d, \{ \vec{x}_{d}^{\vec{A}_{d}}. \phi_{d} \}). \]
    \item The image factorization in $ \int M$ of an arrow $ \lbrack \theta \rbrack : (\vec{b},\{ \vec{y}^{\vec{B}}. \psi \}) \rightarrow  (\vec{a}, \{ \vec{x}^{\vec{A}}. \phi \})$ is given by
    \[ \begin{tikzcd} 
    (\vec{b},\{ \vec{y}^{\vec{B}}. \psi \}) \arrow[]{rr}{\lbrack \theta \rbrack} \arrow[two heads]{rd}{} && (\vec{a}, \{ \vec{x}^{\vec{A}}. \phi \}) \\ & (\vec{a}, \{\exists \vec{y} \theta(\vec{y},\vec{x})\}) \arrow[tail]{ru}{} &  
    \end{tikzcd}\]
    and is easily seen to be pullback stable.
    \item Subobjects in $ \int M$ are arrows of the form $ (\vec{a}, \{ \vec{x}^{\vec{A}}. \phi \}) \hookrightarrow (\vec{a}, \{ \vec{x}^{\vec{A}}. \psi \}) $ with $(\phi \vdash_\mathbb{T} \psi )$. It thus follows at once that lattices of subobjects are frames, as their finite meets (resp. arbitrary joins) are given by 
    \[ (\vec{a}, \{ \vec{x}^{\vec{A}}. \bigwedge_{i \in I} \phi_i \}) \textup{ } (\textrm{resp. } (\vec{a}, \{ \vec{x}^{\vec{A}}. \bigvee_{i \in I} \phi_i \}) ) \]
    for any finite family (resp. arbitrary family)  $\{ (\vec{a}, \{ \vec{x}^{\vec{A}}. \phi_i \}) \hookrightarrow (\vec{a}, \{ \vec{x}^{\vec{A}}. \phi \}) \mid i\in I\}$ of subobjects of a given $(\vec{a}, \{\vec{x}^{\vec{A}}. \phi \})$. 
\end{itemize}\end{proof}

\begin{definition}
We define the \emph{ $ \mathbb{T}$-over-topos at $M$} as the sheaf topos $ \Sh(\int M, J^{\textup{ant}}_{M})$, and denote it by $ \Set[M]$. 
\end{definition}

We shall now define a geometric theory $ \mathbb{T}_M$ whose models in any Grothendieck topos $ \mathcal{G}$ will be all homomorphisms of $\mathbb{T}$-models $ g : N \rightarrow \gamma^*M$, where $\gamma$ is the unique geometric morphism ${\cal G}\to \Set$; this theory will be classified by the $ \mathbb{T}$-over-topos at $M$.

\begin{definition}
Let $\Set_M$ be the language with a sort $ S_{(\vec{a}, \{ \vec{x}^{\vec{A}}. \phi \})}$ for each object $(\vec{a}, \{ \vec{x}^{\vec{A}}. \phi \})$ of the category of elements of $M$ and a function symbol 
    \[ S_{(\vec{a_1}, \{ \vec{x_1}^{\vec{A_1}}. \phi_1 \})} \stackrel{f^{\vec{a_1},\vec{a_2}}_\theta}{\longrightarrow} S_{(\vec{a_2}, \{ \vec{x_2}^{\vec{A_2}}. \phi_2 \})} \]
    for each $ [\theta(\vec{x}_1^{\vec{a_1}},\vec{x}_2^{\vec{a_2}}]_\mathbb{T} : \{ \vec{x}_1^{\vec{a_1}}. \phi_1 \} {\longrightarrow} \{ \vec{x}_2^{\vec{a_2}}. \phi_2 \}  $ such that $\llbracket\theta \rrbracket_M(\vec{a_1}) = \vec{a_2}$.
    
    Let $\mathcal{L}^{c}_M$ be the extension of the language $\cal L$ with a tuple of constant symbols $c_{(\vec{a}, \{ x^{\vec{A}}. \phi \})}$ for each $(\vec{a}, \{ x^{\vec{A}}. \phi \}) \in \int M$. There is a canonical $\mathcal{L}^{c}_M$-structure $M^{c}$ extending $M$, obtained by interpreting each constant by the corresponding element of $M$. 

    We can naturally interpret $\Set_M$ in $\mathcal{L}^c_M$ by replacing, in the obvious way, each variable of sort $S_{(\vec{a}, \{ \vec{x}^{\vec{A}}. \phi \})}$ appearing freely in a formula with the corresponding tuple of constants $c_{(\vec{a}, \{ x^{A}. \phi \})}$ and each function symbol with the corresponding $\mathbb T$-provably functional formula. This yields, for each formula-in-context $ \{ \vec{z}^{\vec{S}}.\psi \} $ over $\Set_M$ a closed formula $\psi^\sharp$.  
    
    Let $\mathbb{T}_M$ be the theory over $\Set_M$ having as axioms all the geometric sequents 
    \[ (\phi \vdash_{\vec{x}^{\vec{S}_{(\vec{a}, \{ \vec{x}^{\vec{A}}. \phi \})}}}  \psi)
    \]
    such that the corresponding sequent
    \[
    (\phi^\sharp \vdash  \psi^\sharp)
    \]
    is valid in $M^{c}$. 
\end{definition}

Given a Grothendieck topos $ \mathcal{G}$ with global section functor $ \gamma$, the model $ M$ is sent in $\mathcal{G}$ to a $\mathbb{T}$-model $ \gamma^*M$ and for each $\{ \vec{x}^A.\phi \}$, one has 
\[  \llbracket \vec{x}^{\vec{A}}.\phi\rrbracket_{\gamma^*M} =  \gamma^*(\llbracket \vec{x}^{\vec{A}}.\phi\rrbracket_M) = \coprod_{ \vec{a}  : 1 \rightarrow  \llbracket \vec{x}^{\vec{A}}.\phi \rrbracket_M} 1 \]
so any global element $ \vec{a}  : 1 \rightarrow  \llbracket \vec{x}^{\vec{A}}.\phi \rrbracket_M$ is sent into a global element  $ \gamma^*(\vec{a}) : 1 \rightarrow \llbracket \vec{x}^{\vec{A}}.\phi \rrbracket_{\gamma^*M}$ in $ \mathcal{G}$. Note that each element of $\llbracket \vec{x}^{\vec{A}}.\phi \rrbracket_{\gamma^*M}$ is counted exactly once in this coproduct. (since coproducts are disjoint in a topos).

\begin{theorem}
Let $ \mathbb{T}$ be a geometric theory and $M$ a model of $ \mathbb{T}$ in $ \Set$. 
\begin{enumerate}
    \item[(i)] The theory $\mathbb{T}_M$ axiomatizes the $\mathbb T$-model homomorphisms to (internalizations of) $M$; that is, for any Grothendieck topos $ \mathcal{G}$ with global section functor $ \gamma : \mathcal{G} \rightarrow \Set$, we have an equivalence of categories 
\[ {\mathbb{T}_M}[{\cal G}] \simeq \mathbb{T}[\mathcal{G}]/\gamma^*M. \]

\item[(ii)] There is a geometric functor 
\[
F_{U}:{\cal C}_{\mathbb{T}_M} \to {\textstyle \int M }
\]
classifying a $\mathbb{T}_M$-model $U$ internal to the category ${\int M }$:
    \[ \begin{array}{ccc}
     S_{(\vec{a}, \{ \vec{x}^{\vec{A}}. \phi \})} & \mapsto & (\vec{a}, \{ \vec{x}^{\vec{A}}. \phi \})  \\
     f^{\vec{a_1},\vec{a_2}}_{\theta}  & \mapsto & [\theta]:(\vec{a_{1}}, \{ \vec{x_{1}}^{\vec{A_{1}}}. \phi_{1} \}) \to (\vec{a_{2}}, \{ \vec{x_{2}}^{\vec{A_{2}}}. \phi_{2} \})
\end{array}  \]

\item[(iii)]
The sheaf topos $\Set[M]=\Sh(\int M, J^{\textup{ant}}_{M})$ is the classifying topos of $ \mathbb{T}_M$; that is, for any Grothendieck topos $\mathcal{G}$ with global section functor $ \gamma : \mathcal{G} \rightarrow \Set$, we have an equivalence of categories 
\[ \textup{\bf Geom}[\mathcal{G}, \Set[M]] \simeq \mathbb{T}[\mathcal{G}]/\gamma^*M. \]

\item[(iv)]  The $\mathbb{T}_M$-model $U$ in $\int M$ as in (ii) is sent by the canonical functor ${\int M}\to \Sh({\int M}, J^{\textup{ant}}_{M})$ to \ac the' universal model of $ \mathbb{T}_M$ inside its classifying topos.

    \item[(v)] There is a full and faithful canonical functor
    \[ \begin{array}{ccc}
    \int M & \stackrel{V}{\rightarrow} & {\cal C}_{\mathbb{T}_M} \\
    (\vec{a}, \{ \vec{x}^{\vec{A}}. \phi \}) & \mapsto & \{x^{S_{(\vec{a}, \{ \vec{x}^{\vec{A}}. \phi \})}}. \top\} \\
     \lbrack\theta\rbrack  & \mapsto & f^{\vec{a_1},\vec{a_2}}_{\theta}
\end{array}  \]

which is a dense (cartesian but not geometric) morphism of sites $({\int M}, J^{\textup{ant}}_{M})\to ({\cal C}_{{\mathbb{T}_M}}, J_{{\mathbb{T}_M}})$ such that $F_{U}\circ V=1_{\int M}$; in particular, $J^{\textup{ant}}_{M}$ is the topology induced by $J_{{\mathbb{T}_M}}$ via $V$ and, $V$ being full and faithful, it is subcanonical.
\end{enumerate}

\end{theorem}

\begin{proof}

The proof will proceed as follows. We shall first establish, for any Grothendieck topos $\cal G$, an equivalence, natural in $\cal G$, between the category $\mathbb{T}[\mathcal{G}]/\gamma^*M$ of $\mathbb T$-model homomorphisms to $\gamma^*M$ and the category $\textbf{Flat}_{J^{\textup{ant}}_{M}}({\int M}, {\cal G})$ of $J^{\textup{ant}}_{M}$-continuous flat functors ${\int M}\to {\cal G}$. Next, we shall establish an equivalence between $\mathbb{T}[\mathcal{G}]/\gamma^*M$ and $\mathbb{T}_M[\mathcal{G}]$ (natural in $\cal G$), obtaining in the process an explicit axiomatization for the theory ${\mathbb T}_{M}$; moreover, we will show that the resulting equivalence
\[
\textup{\bf Flat}_{J^{\textup{ant}}_{M}}({\textstyle\int M}, {\cal G})\simeq \mathbb{T}_M[\mathcal{G}]\simeq \textup{\bf Cart}_{J_{{\mathbb T}_{M}}}({\cal C}_{{\mathbb T}_{M}}, {\cal G})
\]
is induced, on the one hand, by composition with the cartesian cover-preserving functor $V$, and on the other hand, by composition with the geometric functor $F_{U}$. From this we shall deduce that we have an equivalence of toposes
\[
\Sh({\textstyle\int M}, J^{\textup{ant}}_{M})\simeq \Sh({\cal C}_{{\mathbb T}_{M}}, J_{{\mathbb T}_{M}})
\]
whose two functors are induced by the morphisms of sites $V:({\int M}, J^{\textup{ant}}_{M})\to ({\cal C}_{{\mathbb T}_{M}}, J_{{\mathbb T}_{M}})$ and $F_{U}:({\cal C}_{{\mathbb T}_{M}}, J_{{\mathbb T}_{M}}) \to ({\int M}, J^{\textup{ant}}_{M})$. This in turn implies (by Proposition 5.3 \cite{OCdenseness}) that the morphism of sites $V$ is dense and that $J^{\textup{ant}}_{M}$ is the Grothendieck topology on $\int M$ induced by the syntactic topology $J_{{\mathbb T}_{M}}$. Since, as it is easily seen, $F_{U}\circ V\cong 1_{\int M}$, the functor $V$ is full and faithful and therefore the subcanonicity of $J_{{\mathbb T}_{M}}$ entails that of $J^{\textup{ant}}_{M}$. The above equivalence of toposes also implies, by the syntactic construction of \ac the' universal model of a geometric theory inside its classifying topos, that the ${\mathbb T}_{M}$-model $U$ is sent by the canonical functor ${\int M}\to \Sh({\int M}, J^{\textup{ant}}_{M})$ to \ac the' universal model for ${\mathbb T}_{M}$ inside this topos, thus completing the proof of the theorem.

Let $N$ be in $\mathbb{T}[\mathcal{G}]$ and 
\[ N \stackrel{g}{\longrightarrow} \gamma^*M  \]
be a $ \mathcal{L}$-structure homomorphism in $\mathcal{G}$. By the categorical equivalence between models of a geometric theory and cartesian cover-preserving functors on its syntactic site, $g$ is the same thing as a natural transformation 
\[ \begin{tikzcd}
\llbracket \vec{x}^{\vec{A}}.\phi\rrbracket_N \arrow[]{r}{g_{\{ \vec{x}^{\vec{A}}.\phi \}}} & \llbracket \vec{x}^{\vec{A}}.\phi\rrbracket_{\gamma^*(M)}
\end{tikzcd}   \]
(for $\{\vec{x}^{\vec{A}}.\phi\}\in {\cal C}_{\mathbb T}$). 

Note that $\llbracket \vec{x}^{\vec{A}}.\phi\rrbracket_N $ is the disjoint union 
\[ \llbracket \vec{x}^{\vec{A}}.\phi\rrbracket_N = \coprod_{\vec{a} \in \llbracket \vec{x}^{\vec{A}}.\phi\rrbracket_M} N_{ \{ \vec{x}^{\vec{A}}.\phi \}}^{\vec{a}}\]
of the fibers:
\[ \begin{tikzcd}
N_{ \{ \vec{x}^{\vec{A}}.\phi \}}^{\vec{a}} \arrow[phantom, very near start]{rd}{\lrcorner} \arrow[]{r}{} \arrow[]{d}{} & 1 \arrow[]{d}{\gamma^*(\vec{a})} \\ \llbracket \vec{x}^{\vec{A}}.\phi\rrbracket_N \arrow[]{r}{g_{\{ \vec{x}^{\vec{A}}.\phi \}}} & \llbracket \vec{x}^{\vec{A}}.\phi\rrbracket_{\gamma^*(M)}
\end{tikzcd}
\]

Thus $g_{\{ \vec{x}^{\vec{A}}.\phi \}}$ yields a family of objects $(N_{ \{ \vec{x}^{\vec{A}}.\phi \}}^{\vec{a}})_{(\{ \vec{x}^{\vec{A}}.\phi \}, \vec{a}) \in \int M}$ indexed by the category of elements of $M$. 

The naturality of $g$ implies that for any morphism in $\int M$ corresponding to a $ [\theta]$ in $ \mathcal{C}_\mathbb{T}$, one has a unique arrow 
\[ N_{ \{ \vec{x}_1^{\vec{a_1}}.\phi_1 \}}^{\vec{a_1}} \stackrel{\vec{N}^{\vec{a_1},\vec{a_2}}_{\lbrack\theta\rbrack}}{\rightarrow} N_{ \{ \vec{x}_2^{\vec{a_2}}.\phi_2 \}}^{\vec{a_2}} \]
making the following diagram commute:

\[  \begin{tikzcd}[row sep=small, column sep=small]
N_{ \{ \vec{x}_1^{\vec{a_1}}.\phi_1 \}}^{\vec{a_1}} \arrow[dashed]{rd}{\vec{N}^{\vec{a_1},\vec{a_2}}_{\lbrack\theta\rbrack}} \arrow[phantom, very near start]{rdd}{\lrcorner} \arrow[]{rr}{} \arrow[]{dd}{} & & 1  \arrow[]{dd}[near start, swap]{\gamma^*(\vec{a_1})} \arrow[equal]{rd}{} & \\ 
& N_{ \{ \vec{x}_2^{\vec{a_2}}.\phi_2 \}}^{\vec{a_2}} \arrow[phantom, very near start]{rd}{\lrcorner} \arrow[crossing over]{rr}{} \arrow[]{dd}{} & & 1  \arrow[]{dd}{\gamma^*(\vec{a_2})} \\
\llbracket \vec{x}_1^{\vec{a_1}}.\phi_1\rrbracket_N \arrow[]{rd}[swap]{N([\theta])} \arrow[no head]{r}{g_\{ \vec{x}_1^{\vec{a_1}}.\phi_1 \}} & \arrow[]{r}{} & \llbracket \vec{x}_1^{\vec{a_1}}.\phi_1\rrbracket_{\gamma^*(M)} \arrow[]{rd}{\gamma^*M([\theta])} & \\
& \llbracket x_2^{A_2}.\phi_2\rrbracket_N \arrow[]{rr}{g_{\{ \vec{x}_2^{\vec{a_2}}.\phi_2 \}}} & & \llbracket \vec{x}_2^{\vec{a_2}}.\phi_2\rrbracket_{\gamma^*(M)} 
\end{tikzcd} \]

Let us show that $g$ defines a $J^{\textup{ant}}_{M}$-continuous flat functor:
\[ \begin{array}{ccc}
    \int M & \stackrel{\overline{g}}{\rightarrow}& \mathcal{G} \\
    (\vec{a}, \{ \vec{x}^{\vec{A}}. \phi \}) & \mapsto & N_{\{ \vec{x}^{\vec{A}}. \phi \}}^{\vec{a}} \\
     \lbrack\theta\rbrack  & \mapsto & N^{\vec{a_1},\vec{a_2}}_{\lbrack\theta\rbrack}
\end{array}  \]

We have to check that $\overline{g}$ preserves the terminal object and pullbacks, and that it sends $J^{\textup{ant}}_{M}$-covering families to jointly epimorphic families. These are actually consequences of $N$ being a $\mathbb{T}$-model:

\begin{itemize}
    \item For the terminal object, observe that both $ \llbracket [].\top \rrbracket_N $ and $ \llbracket [].\top\rrbracket_{\gamma^*M} $ are equal to the terminal object $1$, whence $ N_{\{ []. \top \}}^* = 1$ too.
    \item For pullbacks, one knows that 
    \[\llbracket \vec{x}_1^{\vec{a_1}},\vec{x}_2^{\vec{a_2}}. \theta_1(\vec{x}_1)=\theta_2(\vec{x}_2) \rrbracket_N \simeq 
\llbracket \vec{x}_1^{\vec{a_1}}. \phi_1 \rrbracket_N \times_{\llbracket  \vec{x}^{\vec{A}}. \phi  \rrbracket_N} \llbracket \vec{x}_2^{\vec{a_2}}. \phi_2 \rrbracket_N \] as $N$ is a model; then for $\vec{a_1} \in  \llbracket \vec{x}_1^{\vec{a_1}}.\phi_1\rrbracket_M $, $\vec{a_2} \in  \llbracket \vec{x}_2^{\vec{a_2}}.\phi_2\rrbracket_M$ such that $\llbracket \theta_1 \rrbracket_M(\vec{a_1})=\vec{c}=\llbracket \theta_2 \rrbracket_M(\vec{a_2})$, one has 
\begin{equation*}
\begin{split}
     &N_{\{ \vec{x}_1^{\vec{a_1}},\vec{x}_2^{\vec{a_2}}. \theta_1(\vec{x}_1)=\theta_2(\vec{x}_2) \}}^{\vec{a_1},\vec{a_2}} = 1 \times^{ \gamma^*(\vec{a_1}),\gamma^*(\vec{a_2})}_{\llbracket \vec{x}_1^{\vec{a_1}},\vec{x}_2^{\vec{a_2}}. \theta_1(\vec{x}_1)=\theta_2(\vec{x}_2) \rrbracket_{\gamma^*M}} \llbracket \vec{x}_1^{\vec{a_1}},\vec{x}_2^{\vec{a_2}}. \theta_1(\vec{x}_1)=\theta_2(\vec{x}_2) \rrbracket_N \\ 
     &\simeq 1 \times^{ \gamma^*(\vec{a_1}),\gamma^*(\vec{a_2})}_{\llbracket \vec{x}_1^{\vec{a_1}},\vec{x}_2^{\vec{a_2}}. \theta_1(\vec{x}_1)=\theta_2(\vec{x}_2) \rrbracket_{\gamma^*M}} ( \llbracket  \vec{x}_1^{\vec{a_1}}. \phi_1 \rrbracket_N \times_{\llbracket   \vec{x}^{\vec{A}}. \phi  \rrbracket_N}
\llbracket \vec{x}_2^{\vec{a_2}}. \phi_2 \rrbracket_N
 ) \\ &\simeq (1 \times^{\gamma^{\ast}(\vec{a_1})}_{\llbracket  \vec{x}_1^{\vec{a_1}}. \phi_1 \rrbracket_{\gamma^*M}} \llbracket  \vec{x}_1^{\vec{a_1}}. \phi_1 \rrbracket_N) \times_{1 \times^{\gamma^*(\vec{c})}_{ \llbracket  \vec{x}^{\vec{A}}. \phi \rrbracket_{\gamma^*M}} \llbracket \vec{x}^{\vec{A}}. \phi \rrbracket_N} (1 \times^{\gamma^{\ast}(\vec{a_2})}_{\llbracket  \vec{x}_2^{\vec{a_2}}. \phi_2 \rrbracket_{\gamma^*M}} \llbracket  \vec{x}_2^{\vec{a_2}}. \phi_2 \rrbracket_N)\\
 &\simeq N^{\vec{a_1}}_{\{ \vec{x}_1^{\vec{a_1}}. \phi_1 \}} \times_{N^{\vec{c}}_{\{  \vec{x}^{\vec{A}}. \phi  \}}} N^{\vec{a_2}}_{\{ \vec{x}_2^{\vec{a_2}}. \phi_2 \}}.
\end{split}
\end{equation*}

\item For $J^{\textup{ant}}_{M}$-continuity: let 
\[ ([\theta_i] : (\vec{b},\{ \vec{x}_i^{A_i}.\phi_i \}) {\rightarrow} (\vec{a},\{  x^{\vec{A}}. \phi  \}))_{\vec{b} \in \langle \llbracket\theta_i \rrbracket_M \rangle_{i \in I}^{-1}(\vec{a}) } \] be a family in ${\cal B}_{\mathbb T}^{\textup{ant}}$. As $\gamma^*M$ and $ N$ are $\mathbb{T}$-models, they both send $ ([\theta_i])_{i \in I}$ to epimorphic families in $\mathcal{G}$. On the other hand, one can express the fiber of $a$ along the coproduct map $\langle \llbracket\theta_i \rrbracket_M \rangle_{i \in I}^{-1}$ as \[ \coprod_{\langle \llbracket\theta_i \rrbracket_M \rangle_{i \in I}^{-1}(\vec{a})} 1 \simeq \coprod_{i \in I}  \llbracket\theta_i \rrbracket_M^{-1}(\vec{a})  \rightarrow \coprod_{i \in I} \llbracket \vec{x}_i^{\vec{A}_i}.\phi_i \rrbracket_M. \] 

Moreover, this pullback is preserved by $ \gamma^*$, which sends it to 
    \[ \gamma^*( \langle \llbracket\theta_i \rrbracket_M \rangle_{i \in I}^{-1}(\vec{a})) \simeq \coprod_{i \in I} \gamma^*( \llbracket\theta_i \rrbracket_M^{-1}(\vec{a}) ) \simeq \coprod_{\langle \llbracket\theta_i \rrbracket_M \rangle_{i \in I}^{-1}(\vec{a})} \gamma^*(1) \simeq \coprod_{\langle \llbracket\theta_i \rrbracket_M \rangle_{i \in I}^{-1}(\vec{a})} 1 \]
which is the $  \langle \llbracket\theta_i \rrbracket_M \rangle_{i \in I}^{-1}(\vec{a})$-indexed coproduct of the terminal object of $ \mathcal{G}$.    
Then by the stability of coproducts under pullbacks one has 
     \begin{equation*}
        \begin{split}
            \gamma^*(\langle \llbracket\theta_i \rrbracket_M \rangle_{i \in I}^{-1}(\vec{a})) \times_{ \llbracket  \vec{x}^{\vec{A}}. \phi  \rrbracket_{\gamma^*M}} \llbracket  \vec{x}^{\vec{A}}. \phi  \rrbracket_N 
            &\simeq (\coprod_{ \vec{b} \in \langle \llbracket\theta_i \rrbracket_M \rangle_{i \in I}^{-1}(\vec{a}) } 1) \times_{ \llbracket  \vec{x}^{\vec{A}}. \phi  \rrbracket_{\gamma^*M}} \llbracket  \vec{x}^{\vec{A}}. \phi  \rrbracket_N  \\ &= \coprod_{\vec{b} \in \langle \llbracket\theta_i \rrbracket_M \rangle_{i \in I}^{-1}(\vec{a})} N^{\vec{b}}_{\{ \vec{x}_i^{\vec{A}_i}. \phi_i \}}. 
        \end{split}
    \end{equation*} 
   
    So in the diagram
    \[ \begin{tikzcd}[row sep=small, column sep=small]
\coprod_{\vec{b} \in \langle \llbracket\theta_i \rrbracket_M \rangle_{i \in I}^{-1}(\vec{a})} N^{\vec{b}}_{\{ \vec{x}_i^{\vec{A}_i}. \phi_i \}} \arrow[]{rr}{} \arrow[]{rd}{} \arrow[]{dd}{} & & \gamma^*(\langle \llbracket\theta_i \rrbracket_M\rangle_{i \in I}^{-1}(\vec{a})) \arrow[]{rd} \arrow[]{dd}{} \\
& N^{\vec{a}}_{\{ \vec{x}^{\vec{A}}. \phi \}} \arrow[]{dd}{} \arrow[crossing over]{rr}{} & & 1 \arrow[]{dd}{\gamma^*(\vec{a})} \\ 
\coprod_{i \in I} \llbracket \vec{x}_i^{\vec{A}_i}.\phi_i\rrbracket_N \arrow[two heads]{rd} \arrow[no head]{r}{\coprod_{i\in I} g_{\{ \vec{x}_i^{\vec{A}_i}.\phi_i \}}} & \arrow[]{r}{} & \coprod_{i \in I} \llbracket \vec{x}_i^{\vec{A}_i}.\phi_i\rrbracket_{\gamma^*(M)} \arrow[two heads]{rd} & 
\\
& \llbracket \vec{x}^{\vec{A}}.\phi\rrbracket_N \arrow[]{rr}{g_{\{ \vec{x}^{\vec{A}}.\phi \}}} & & \llbracket \vec{x}^{\vec{A}}.\phi\rrbracket_{\gamma^*(M)} 
    \end{tikzcd}\]
    the front, right and back squares are pullbacks, whence the left square is a pullback too: but this forces the upper left arrow 
    \[ \coprod_{\vec{b} \in \langle \llbracket\theta_i \rrbracket_M \rangle_{i \in I}^{-1}(\vec{a})} N^{\vec{b}}_{\{ \vec{x}_i^{\vec{A}_i}. \phi_i \}} \rightarrow  N^{\vec{a}}_{\{ \vec{x}^{\vec{A}}. \phi \}} \]
    to be an epimorphism by the stability of epimorphisms under pullback in $\mathcal{G}$. 
\end{itemize}

The data of the $N^a_{\{ \vec{x}^A. \phi \}}$'s with their transitions morphisms define a $\Set_{M}$-structure $S_{g}$: 
\\
    \[ \begin{array}{ccc}
     S_{(\vec{a}, \{ \vec{x}^{\vec{A}}. \phi \})} & \mapsto & N_{\{ \vec{x}^{\vec{A}}. \phi \}}^{\vec{a}}  \\
     f^{\vec{a_1},\vec{a_2}}_{\theta}  & \mapsto & N^{\vec{a_1},\vec{a_2}}_{\lbrack\theta\rbrack}:  N_{ \{ \vec{x}_1^{\vec{a_1}}.\phi_1 \}}^{\vec{a_1}} \to N_{ \{ \vec{x}_2^{\vec{a_2}}.\phi_2 \}}^{\vec{a_2}}
\end{array}  \]

This structure is actually a ${\mathbb T}_{M}$-model. Indeed, this follows at once from the fact that the interpretation in $S_{g}$ of any geometric formula $\psi$ over the signature $S_{M}$ can be expressed as a pullback of the interpretation of the corresponding formula $\psi^{\sharp}$ in the ${\cal L}^{c}_{M}$-structure $M^{c}$, as in the following diagram:

\[\begin{tikzcd}
	{[[z_{1}^{S_{(\vec{a_1}, \{ \vec{x_1}^{\vec{A_1}}. \phi_1 \})}}, \ldots , z_{n}^{S_{(\vec{a}_n, \{ \vec{x_n}^{\vec{A_n}}. \phi_n \})}}. \psi]]_{S_{g}}} && {[[[]. \psi^\sharp]]_{M^{c}}} \\
	{[[z_{1}^{S_{(\vec{a_1}, \{ \vec{x_1}^{\vec{A_1}}. \phi_1 \})}}, \ldots , z_{n}^{S_{(\vec{a}_n, \{ \vec{x_n}^{\vec{A_n}}. \phi_n \})}}. \top]]_{S_{g}}} && 1 \\
	\\
	{[[\vec{x_{1}}^{\vec{A_{1}}}, \ldots, \vec{x_{n}}^{\vec{A_{n}}}.\top]]_{N}} && {[[\vec{x_{1}}^{\vec{A_{1}}}, \ldots, \vec{x_{n}}^{\vec{A_{n}}}. \top]]_{\gamma^{\ast}(M)}}
	\arrow[from=1-1, to=1-3]
	\arrow[tail, from=1-3, to=2-3]
	\arrow["{g\vec{A_1}\times \ldots \times g\vec{A_n}}", from=4-1, to=4-3]
	\arrow["\lrcorner"{anchor=center, pos=0.125}, draw=none, from=1-1, to=4-3]
	\arrow["{<\vec{a_1}, \ldots, \vec{a_n}>}"{description}, tail, from=2-3, to=4-3]
	\arrow[tail, from=1-1, to=2-1]
	\arrow[tail, from=2-1, to=4-1]
	\arrow[from=2-1, to=2-3]
	\arrow["\lrcorner"{anchor=center, pos=0.125}, draw=none, from=2-1, to=4-3]
\end{tikzcd}\]

This analysis shows that a simple axiomatization for the theory ${\mathbb T}_{M}$ may be obtained by phrasing in logical terms the property that the functor $V:{\int M}\to {\cal C}_{{\mathbb T}_{M}}$ in the statement of the theorem be cartesian and cover-preserving. Recalling the well-known characterizations of pullbacks and terminal objects in the internal language of a topos, this leads to the following axioms for ${\mathbb T}_{M}$: 
\[
\big{(}\top \vdash_{[]} (\exists x^{S_{(\vec{a}, \{[]. \top \})}}) \top\big{)};
\]
\[
\big{(}\top \vdash_{x^{S_{(\vec{a}, \{[]. \top \})}}, x'^{S_{(\vec{a}, \{[]. \top \})}}} (x=x')\big{)};
\]
\[
\big{(}\top \vdash_{z^{S_{({\vec{a_1},\vec{a_2}}}, \{\vec{x_1'}^{\vec{a_1}},\vec{x_2'}^{\vec{a_2}}. \theta_1(\vec{x_1'})=\theta_2(\vec{x_2'})\})}}   f^{\vec{a_{1}}, \vec{c}}_{\theta_1}(f^{(\vec{a_1}, \vec{a_2}), \vec{a_{1}}}_{\vec{x_{1}}=\vec{x_{1}'}}(z))= f^{\vec{a_{2}}, \vec{c}}_{\theta_1}(f^{(\vec{a_1}, \vec{a_2}), \vec{a_{2}}}_{\vec{x_{2}}=\vec{x_{2}'}}(z)))\big{)},
\]
\[
\big{(}f^{(\vec{a_1}, \vec{a_2}), \vec{a_{1}}}_{\vec{x_{1}}=\vec{x_{1}'}}(z)=f^{(\vec{a_1}, \vec{a_2}), \vec{a_{1}}}_{\vec{x_{1}}=\vec{x_{1}'}}(z')) \wedge (f^{(\vec{a_1}, \vec{a_2}), \vec{a_{2}}}_{\vec{x_{2}}=\vec{x_{2}'}}(z)=f^{(\vec{a_1}, \vec{a_2}), \vec{a_{2}}}_{\vec{x_{2}}=\vec{x_{2}'}}(z')) \vdash_{z, z'} z=z'\big{)},
\]
\[
\big{(}f^{\vec{a_{1}}, \vec{c}}_{\theta_1}(y_1)=f^{\vec{a_{2}}, \vec{c}}_{\theta_2}(y_2) \vdash_{y_1^{S_{(\vec{a_1}, \{ \vec{x_1}^{\vec{A_1}}. \phi_1 \})}}, y_2^{S_{(\vec{a_2}, \{ \vec{x_2}^{\vec{A_2}}. \phi_2 \}) } }} (\exists z)((f^{(\vec{a_1}, \vec{a_2}), \vec{a_{1}}}_{\vec{x_{1}}=\vec{x_{1}'}}(z)=y_1) \wedge (f^{(\vec{a_1}, \vec{a_2}), \vec{a_{2}}}_{\vec{x_{2}}=\vec{x_{2}'}}(z)=y_2))\big{)} 
\]
for any $\vec{a_1} \in  \llbracket \vec{x}_1^{\vec{a_1}}.\phi_1\rrbracket_M $, $\vec{a_2} \in  \llbracket \vec{x}_2^{\vec{a_2}}.\phi_2\rrbracket_M$ such that $\llbracket \theta_1 \rrbracket_M(\vec{a_1})=\vec{c}=\llbracket \theta_2 \rrbracket_M(\vec{a_2})$; 
\[
\big{(} \top \vdash_{y^{S_(\vec{a}, {\{ \vec{x}^{\vec{A}}. \phi \}}))}} 
\bigvee_{i\in I, \vec{b_i} \in \llbracket\theta_i \rrbracket_M^{-1}(\vec{a})} (\exists z^{S_{(\vec{b_i},\{\vec{x}_i^{\vec{A}_i}. \phi_i \})}} ) (f_{\theta_i}^{\vec{b_i}, \vec{a}}(z)=y)\big{)}
\]
for any family $\{[\theta_i]:  \{ \vec{x_i}^{\vec{A_i}}. \phi_i \} \to \{ \vec{x}^{\vec{A}}. \phi \} \mid i\in I \}\in {\cal B}_{\mathbb T}(\{ \vec{x}^{\vec{A}}. \phi \})$ and any $\vec{a}\in \llbracket \vec{x}^{\vec{A}}.\phi\rrbracket_M$.

The first two axioms express the fact that $N^{\ast}_{\{[]. \top\}}=1$, the following three express the fact that for any $\vec{a_1} \in  \llbracket \vec{x}_1^{\vec{a_1}}.\phi_1\rrbracket_M $, $\vec{a_2} \in  \llbracket \vec{x}_2^{\vec{a_2}}.\phi_2\rrbracket_M$ such that $\llbracket \theta_1 \rrbracket_M(\vec{a_1})=\vec{c}=\llbracket \theta_2 \rrbracket_M(\vec{a_2})$, we have a pullback square
\[\begin{tikzcd}
	{N_{\{ \vec{x_1'}^{\vec{a_1}},\vec{x_2'}^{\vec{a_2}}. \theta_1(\vec{x_1'})=\theta_2(\vec{x_2'}) \}}^{\vec{a_1},\vec{a_2}}} && {N^{\vec{a_1}}_{\{ \vec{x}_1^{\vec{a_1}}. \phi_1 \}}} \\
	\\
	{N^{\vec{a_2}}_{\{ \vec{x}_2^{\vec{a_2}}. \phi_2 \}}} && {N^{\vec{c}}_{\{  \vec{x}^{\vec{A}}. \phi  \}}}
	\arrow["{N^{\vec{a_{1}}, \vec{c}}_{\theta_1}}", from=1-3, to=3-3]
	\arrow["{N^{\vec{a_{2}}, \vec{c}}_{\theta_2}}", from=3-1, to=3-3]
	\arrow["{N^{(\vec{a_1}, \vec{a_2}), \vec{a_{2}}}_{\vec{x_{2}}=\vec{x_{2}'}}}", from=1-1, to=3-1]
	\arrow["{N^{(\vec{a_1}, \vec{a_2}), \vec{a_{1}}}_{\vec{x_{1}}=\vec{x_{1}'}}}", from=1-1, to=1-3]
\end{tikzcd}\]

and the last one corresponds to the property of $V$ being cover-preserving.

It is immediate to see that the $J_{{\mathbb T}_{M}}$-continuous cartesian functor $F_{S_{g}}: {\cal C}_{{\mathbb T}_{M}}\to {\cal G}$ corresponding to the ${\mathbb T}_{M}$-model $S_{g}$ is given by $\overline{g}\circ F_{U}$ and that, conversely, the flat functor $\overline{g}$ can be recovered from the ${\mathbb T}_{M}$-model $S_{g}$ as the composite $F_{S_{g}}\circ V$.

Let us now show that, in the other direction, given families 
\[
{\big{(} N^{\vec{a}}_{ \{\vec{x}_1^{\vec{A}}.\phi\}}\big{)}}_{ (\vec{a}, \{ \vec{x}^{\vec{A}}.\phi \})\in {\int M}}
\]
and 
\[
\big{(}f^{\vec{a_1},\vec{a_2}}_{\theta}: N_{ \{ \vec{x}_1^{\vec{a_1}}.\phi_1 \}}^{\vec{a_1}} \to N_{ \{ \vec{x}_2^{\vec{a_2}}.\phi_2 \}}^{\vec{a_2}} \big{)}_{[\theta]:\{\vec{x_1}. \phi_1\} \to \{\vec{x_1}. \phi_1\}  \in \mathcal{C}_\mathbb{T}\textup{ with } [[\theta]]_{M}(\vec{a_1})=\vec{a_2}}
\]
respectively of objects and arrows in $\mathcal{G}$ defining a $J^{\textup{ant}}_{M}$-continuous flat (equivalently, cartesian) functor $G:{\int M}\to {\cal G}$, or, equivalently a ${\mathbb T}_{M}$-model in $\cal G$, we can associate with it a $\mathbb T$-model $N$ in $\cal G$ and a homomorphism of $\mathbb T$-models $g:N \to \gamma^{\ast}(M)$. 

First, for each each $\{ \vec{x}^A.\phi \}$ and $\vec{a} \in \llbracket \vec{x}^{A}.\phi\rrbracket_M $, we set $g_{\{ \vec{x}^A.\phi\}}^{\vec{a}}$ equal to the composite arrow
\[ 
\begin{tikzcd}
N_{\{\vec{x}^{\vec{A}}.\phi\}}^{\vec{a}} \arrow[]{r}{!} \arrow[]{rd}[swap]{g_{\{ \vec{x}^A.\phi\}}^{\vec{a}}} & 1 \arrow[]{d}{\gamma^*(\vec{a})} \\ & \llbracket \vec{x}^{\vec{A}}.\phi\rrbracket_{\gamma^*(M)}  
\end{tikzcd}
\]

Then we define, for each object $\{ \vec{x}^A.\phi \}$ of $ \mathcal{C}_\mathbb{T}$, $g_{\{ \vec{x}^{ \vec{A}}.\phi \}}$ as the arrow determined by the universal property of the coproduct as in the following diagrams (where the vertical arrows are the canonical coproduct inclusions):
\[ \begin{tikzcd} 
N_{ \{ \vec{x}^{ \vec{A}}.\phi \}}^{ \vec{a}} \arrow[]{rd}{g_{ \{ \vec{x}^{ \vec{A}}.\phi \}}^{ \vec{a}}} \arrow[]{d}{} & \\ 
\underset{\vec{a} \in \llbracket \vec{x}^A.\phi\rrbracket_M}{\coprod} N_{ \{ \vec{x}^{ \vec{A}}.\phi \}}^{ \vec{a}} \arrow[dashed]{r}[swap]{g_{\{ \vec{x}^{ \vec{A}}.\phi \}}} & \llbracket \vec{x}^{{ \vec{A}}}.\phi\rrbracket_{\gamma^*(M)}
\end{tikzcd}\]

We need to ensure that: \begin{itemize}
    \item we have a $J_\mathbb{T}$-continuous cartesian functor: 
    \[ \begin{array}{rcl}
        \mathcal{C}_\mathbb{T} & \stackrel{F_{N}}{\longrightarrow} & \mathcal{G}  \\
        \{ \vec{x}^{ \vec{A}}.\phi \} & \mapsto & N_{\{ \vec{x}^{ \vec{A}}.\phi \}} = \underset{a \in \llbracket x^{\vec{A}}.\phi\rrbracket_M}{\coprod} N_{ \{ \vec{x}^{ \vec{A}}.\phi \}}^{ \vec{a}} \\
        \theta:\{\vec{x_1}^{\vec{A_1}}. \phi_1\} \to \{\vec{x_2}^{\vec{A_2}}. \phi_2\} & \mapsto & \underset{\vec{a_1} \in \llbracket x_1^{\vec{A_1}}.\phi\rrbracket_M}{\coprod} f^{\vec{a_1}, [[\theta]]_{M}(\vec{a_1})}:  \underset{\vec{a_1} \in \llbracket x_1^{\vec{A_1}}.\phi\rrbracket_M}{\coprod}     N_{ \{ \vec{x_1}^{ \vec{A_1}}.\phi_1 \}}^{ \vec{a_1}} \to \underset{\vec{a_2} \in \llbracket x_2^{\vec{A_1}}.\phi\rrbracket_M}{\coprod} N_{ \{ \vec{x_2}^{ \vec{A_2}}.\phi_2 \}}^{\vec{a_2}}
    \end{array} \]
    and hence a $\mathbb T$-model $N$ in $\cal G$;
    
    \item we have a natural transformation 
    \[ g = (g_{\{ \vec{x}^{ \vec{A}}.\phi \}})_{\{ \vec{x}^{ \vec{A}}.\phi \} \in \mathcal{C}_\mathbb{T}} : N \rightarrow \gamma^*M. \]
\end{itemize}

This amounts to the following conditions:

\begin{itemize}
    \item If $ \{ []. \top \} $ is the terminal object of the syntactic site, then its interpretation in $M$ has exactly one global element $ 1 \rightarrow \llbracket [], \top  \rrbracket_M = 1 $ so requiring $F_{N}$ to preserve the terminal object is equivalent to demanding $ N^*_{ \{ []. \top \} } = 1$, which is ensured by the fact that the functor $G$ preserves the terminal object by our hypotheses.
    
    \item Pullbacks in $ \mathcal{C}_\mathbb{T}$ 
    \[ \begin{tikzcd}
    \{ \vec{x_1}^{\vec{A_1}},\vec{x_2}^{\vec{A_2}}. \theta_1(\vec{x_1})=\theta_2(\vec{x_2}) \}  \arrow[phantom, very near start]{rd}{\lrcorner} \arrow[]{r}{} \arrow[]{d}{} & \{ \vec{x_1}^{\vec{A_1}}. \phi_1 \} \arrow[]{d}{\lbrack \theta_1 \rbrack} \\  \{ \vec{x_2}^{\vec{A_2}}. \phi_2 \} \arrow[]{r}{\lbrack \theta_2 \rbrack} & \{  \vec{x}^{\vec{A}}. \phi  \}
    \end{tikzcd}\]
    are sent by $F_{M}$ to pullbacks in $\Set$
    \[ \begin{tikzcd}
   \llbracket  \vec{x_1}^{\vec{A_1}},\vec{x_2}^{\vec{A_2}}. \theta_1(\vec{x_1})=\theta_2(\vec{x_2}) \rrbracket_M \arrow[]{r}{} \arrow[phantom, very near start]{rd}{\lrcorner}  \arrow[]{d}{} & \llbracket  \vec{x_1}^{\vec{A_1}}. \phi_1 \rrbracket_M \arrow[]{d}{\llbracket \theta_1 \rrbracket_M} \\  \llbracket \vec{x_2}^{\vec{A_2}}. \phi_2 \rrbracket_M \arrow[]{r}{\llbracket \theta_2 \rrbracket_M} & \llbracket   \vec{x}^{\vec{A}}. \phi  \rrbracket_M,
    \end{tikzcd}\]
    where elements of $ \llbracket  \vec{x_1}^{\vec{A_1}},\vec{x_2}^{\vec{A_2}}. \theta_1(\vec{x_1})=\theta_2(\vec{x_2}) \rrbracket_M$ are pairs $ (\vec{a_1},\vec{a_2}) : 1 \rightarrow \llbracket \vec{x_1}^{\vec{A_1}}. \phi_1 \rrbracket_M \times 
\llbracket \vec{x_2}^{\vec{A_2}}. \phi_2 \rrbracket_M  $ such that $ \llbracket \theta_1 \rrbracket_M (\vec{a_1}) = \vec{a} = \llbracket \theta_2 \rrbracket_M (\vec{a_2})$. Now, if all the squares of the form 
\[ \begin{tikzcd}
    N_{\{ \vec{x_1}^{\vec{A_1}},\vec{x_2}^{\vec{A_2}}. \theta_1(\vec{x_1})=\theta_2(\vec{x_2}) \}}^{\vec{a_1},\vec{a_2}} \arrow[]{r}{} \arrow[]{d}{} & N_{\{\vec{x_1}^{\vec{A_1}}. \phi_1 \}}^{\vec{a_1}} \arrow[]{d}{N^{\vec{a_1},\vec{a}}_{\lbrack \theta_1 \rbrack}} \\  N_{\{\vec{x_2}^{\vec{A_2}}. \phi_2 \}}^{\vec{a_2}} \arrow[]{r}{N^{\vec{a_2},\vec{a}}_{\lbrack \theta_2 \rbrack}} & N_{\{\vec{x}^{\vec{A}}. \phi  \}}^{\vec{a}}
    \end{tikzcd}\]
are pullbacks, which is the case since the functor $G$ preserves pullbacks by our hypotheses, then, by the stability of coproducts along pullbacks, we have 
\begin{equation*}
    \begin{split}
      N_{\{ \vec{x_1}^{\vec{A_1}},\vec{x_2}^{\vec{A_2}}. \theta_1(\vec{x_1})=\theta_2(\vec{x_2}) \}} &= \coprod_{\llbracket  \vec{x_1}^{\vec{A_1}},\vec{x_2}^{\vec{A_2}}. \theta_1(\vec{x_1})=\theta_2(\vec{x_2}) \rrbracket_M}  N_{\{ \vec{x_1}^{\vec{A_1}},\vec{x_2}^{\vec{A_2}}. \theta_1(\vec{x_1})=\theta_2(\vec{x_2}) \}}^{\vec{a_1},\vec{a_2}} \\
      &=  \coprod_{\llbracket  \vec{x_1}^{\vec{A_1}},\vec{x_2}^{\vec{A_2}}. \theta_1(\vec{x_1})=\theta_2(\vec{x_2}) \rrbracket_M}  N_{\{\vec{x_1}^{\vec{A_1}}. \phi_1 \}}^{\vec{a_1}}   \times_{N_{\{\vec{x}^{\vec{A}}. \phi  \}}^{\vec{a}} }  N_{\{ \vec{x_2}^{\vec{A_2}}. \phi_2 \}}^{\vec{a_2}}  \\  
      &= \coprod_{\llbracket  \vec{x_1}^{\vec{A_1}}. \phi_1 \rrbracket_M} N_{\{ \vec{x_1}^{\vec{A_1}}. \phi_1 \}}^{\vec{a_1}}  \times_{\underset{\vec{a} \in \llbracket   \vec{x}^{\vec{A}}. \phi  \rrbracket_M}{\coprod} N_{\{  \vec{x}^{\vec{A}}. \phi  \}}^{\vec{a}} }  \coprod_{\llbracket  \vec{x_2}^{\vec{A_2}}. \phi_2 \rrbracket_M} N_{\{ \vec{x_2}^{\vec{A_2}}. \phi_2 \}}^{\vec{a_2}}\\
      &= N_{\{ \vec{x_1}^{\vec{A_1}}. \phi_1 \}} \times_{N_{\{  \vec{x}^{\vec{A}}. \phi  \}}} N_{\{ \vec{x_2}^{\vec{A_2}}. \phi_2 \}}.
    \end{split}
\end{equation*}

\item $ J_\mathbb{T}$-continuity: given a small $J_{\mathbb T}$-cover $( \lbrack \theta_i \rbrack :  \{\vec{x_i}^{\vec{A_i}}. \phi_i \} \rightarrow \{\vec{x}^{\vec{A}}. \phi \} )_{i \in I}$, since $M$ is a $\mathbb T$-model, the corresponding functor $F_{M}:{\cal C}_{\mathbb T} \to {\cal G}$ sends it to a jointly epimorphic family  $( \llbracket \theta_i \rrbracket_M :  \llbracket \vec{x_i}^{\vec{A_i}}. \phi_i \rrbracket_M \rightarrow \llbracket \vec{x}^{\vec{A}}. \phi \rrbracket_M )_{i \in I}$; so each $\vec{a} \in \llbracket x^{A}. \phi \rrbracket_M$ has an antecedent $\vec{b} \in \llbracket \vec{x_i}^{\vec{A_i}}. \phi_i \rrbracket_M$ for some $i \in I$. Therefore for each $\vec{a}\in \llbracket \vec{x}. \phi \rrbracket$ the cocone of fibers  
\[ 
\big{(}N_{\{\vec{x_i}^{\vec{A_i}}. \phi_i\}}^{\vec{b}} \stackrel{N^{\vec{b},\vec{a}}_{\lbrack \theta_i \rbrack}}{\longrightarrow} N_{\{ \vec{x}^{\vec{A}}. \phi \}}^{\vec{a}}\big{)}_{i \in I, \; \llbracket \theta_i \rrbracket_M(\vec{b})=\vec{a}} 
\]
in each of its antecedents is jointly epimorphic in $\mathcal{G}$ if and only if the coproduct of fibers \[\big{(}   N_{\{\vec{x_i}^{\vec{A_i}}. \phi_i \}} \stackrel{N_{\lbrack \theta_i \rbrack}}{\longrightarrow} N_{\{\vec{x}^{\vec{A}}. \phi \}}\big{)}_{i\in I}  \] is. But this exactly amounts to requiring the following functor to be $ J^{\textup{ant}}_{M}$-continuous:
\[ \begin{array}{ccc}
    \int M & \stackrel{G}{\rightarrow}& \mathcal{G} \\
    (\vec{a}, \{ \vec{x}^{\vec{A}}. \phi \}) & \mapsto & N_{\{ \vec{x}^{\vec{A}}. \phi \}}^{\vec{a}} \\
     \lbrack\theta\rbrack  & \mapsto & N^{\vec{a_1},\vec{a_2}}_{\lbrack\theta\rbrack}
\end{array}  \]

\end{itemize}
The naturality of $g$ follows immediately from the definition of the functor $F_{N}$ on the arrows of ${\cal C}_{\mathbb T}$.

To conclude our proof of the categorical equivalence between flat $J^{\textup{ant}}_{M}$-continuous functors on ${\int M}$ and $\mathbb T$-model homomorphisms to $\gamma^{\ast}(M)$, we have to check that the two functors defined above are mutually quasi-inverse. For the construction starting from a homomorphism of $ \mathcal{L}$-structures between $ \mathbb{T}$-models $ g : N \rightarrow \gamma^*M$, observe that as the codomain of $g$ decomposes as the coproduct of all elements $\llbracket \vec{x}^{\vec{A}}.\phi\rrbracket_{\gamma^*M} =
\coprod_{ \vec{a}  : 1 \rightarrow  \llbracket \vec{x}^{\vec{A}}.\phi \rrbracket_M} 1$, by the stability of coproducts under pullbacks one has 
\[ N_{ \{ \vec{x}^{ \vec{A}}.\phi \}} \simeq
\underset{\vec{a} \in \llbracket \vec{x}^A.\phi\rrbracket_M}{\coprod} N_{ \{ \vec{x}^{ \vec{A}}.\phi \}}^{ \vec{a}}. \]
For the converse process, if one starts with a $J^{\textup{ant}}_{M}$-continuous flat functor $ N_{(-)} : \int M \rightarrow \mathcal{G}$ and defines, for each $ \{ \vec{x}^{ \vec{A}}.\phi \}$, $N_{ \{ \vec{x}^{ \vec{A}}.\phi \}}$ as the above coproduct, then pulling it back along $ \gamma^*(\vec{a}): 1 \rightarrow \llbracket  \vec{x}^{ \vec{A}}.\phi  \rrbracket_{ \gamma^*M}  $, one has (again, by the stability of coproducts under pullback)
\[ \gamma^*(\vec{a})^*( \underset{\vec{b} \in \llbracket \vec{x}^A.\phi\rrbracket_M}{\coprod} N_{ \{ \vec{x}^{ \vec{A}}.\phi \}}^{ \vec{b}}) \simeq \underset{\vec{b} \in \llbracket \vec{x}^A.\phi\rrbracket_M}{\coprod} \gamma^*(\vec{a})^*(N_{ \{ \vec{x}^{ \vec{A}}.\phi \}}^{ \vec{b}}). \]
For any element $\vec{b}$ of $\llbracket \vec{x}^{ \vec{A}}.\phi  \rrbracket_{M}$, we have the following pullback squares:
\[ 
\begin{tikzcd}
\gamma^*(\vec{a})^*(N_{ \{ \vec{x}^{ \vec{A}}.\phi \}}^{ \vec{b}}) \arrow[d] \arrow[r] \arrow[rd, "\lrcorner", very near start, phantom] & {1 \times^{\gamma^*(\vec{a}), \gamma^*(\vec{b})}_{\llbracket  \vec{x}^{ \vec{A}}.\phi  \rrbracket_{ \gamma^*M}} 1} \arrow[d] \arrow[r] \arrow[rd, "\lrcorner", very near start, phantom] & 1 \arrow[d, "\gamma^*(\vec{a})"]                                                 \\
N_{ \{ \vec{x}^{ \vec{A}}.\phi \}}^{ \vec{b}} \arrow[r]                                            & 1 \arrow[r, "\gamma^*(\vec{b})"]                                                                                                                                        & \llbracket  \vec{x}^{ \vec{A}}.\phi\rrbracket_{ \gamma^*M}
\end{tikzcd} \]
Now, there are two possible values for the pullback on the right-hand side:
\[ 
\begin{tikzcd}[column sep= huge]
0 \arrow[d] \arrow[r] \arrow[rd, "\lrcorner \atop \textrm{ whenever } \vec{a_1} \neq \vec{a_2}", phantom] & 1 \arrow[d, "\gamma^\ast(\vec{a_2})"]                                       &  & 1 \arrow[d, equal] \arrow[r, equal] \arrow[rd, "\lrcorner \atop \textrm{ whenever } \vec{a_1} = \vec{a_2}", phantom] & 1 \arrow[d, "\gamma^{\ast}(\vec{a_2})"]                                       \\
1 \arrow[r, "\gamma^{\ast}(\vec{a_1})"]                                                                            & \llbracket \vec{x}^{\vec{A}}.\phi \rrbracket_{\gamma^*M} &  & 1 \arrow[r, "\gamma^{\ast}(\vec{a_1})"]                                                                         & \llbracket \vec{x}^{\vec{A}}.\phi \rrbracket_{\gamma^*M}
\end{tikzcd} \]
So the middle pullback is the initial object whenever $\vec{a} \neq \vec{b}$, whence $\gamma^*(\vec{a})^*(N_{ \{ \vec{x}^{ \vec{A}}.\phi \}}^{ \vec{b}}) \cong 0 $; on the other hand, $\gamma^*(\vec{a})^*(N_{ \{ \vec{x}^{ \vec{A}}.\phi \}}^{ \vec{b}}) \cong N_{ \{ \vec{x}^{ \vec{A}}.\phi \}}^{ \vec{a}}$ whenever $\vec{a}=\vec{b}$. This clearly implies our thesis.
\end{proof}

\begin{remark}
We have made use of specific properties of $\Set$ at several steps of the above constructions and proofs: 
\begin{itemize}
    \item In the definition of the antecedent topology, we only had to consider global elements of the sets $ \llbracket \vec{x}^{\vec{A}}.\phi\rrbracket_{M}$. This is because $1$ is a generator of $ \Set$, so that generalized elements would just be coproducts of global elements.
    
    \item As a consequence, the very notion of antecedent element is simplified. As we shall see in section 4, considering antecedents of generalized elements gives rise to complications when considering jointly epimorphic families, as antecedents may be indexed by other objects of the topos than the domain of the generalized element we search antecedents of. In this case, we just had to consider the global elements of the fiber of a global element; in other words, the antecedent topology exists already in the comma category $(1 \downarrow F_{M})$; in the general case, it will be scattered on the fibers of a comma $(\y \downarrow F_{M})$, for $\y $ the Yoneda embedding of a small, cartesian subcanonical site for the given topos.
    
    \item We also used that $1$ was indecomposable to retrieve global elements of the fibers of a jointly surjective family. This is not anymore a valid argument in an arbitrary Grothendieck topos.  
\end{itemize}
\end{remark}

\section{An interlude on stacks}

We will turn in the next section to the construction of the over-topos at a model in the general case, where the base topos in which the given model lives is an arbitrary Grothendieck topos. We chose to treat this general case separately as it requires Giraud's theory of the classifying topos of a stack, while the set-valuated case required more conventional tools. In this section we present a number of results on stacks that we shall need in our analysis.

Given a pseudofunctor ${\mathbb I}:{\cal C}^{\textup{op}}\to \textup{Cat}$, we will denote by $\pi_{\mathbb I}$ the canonical projection functor ${\cal G}({\mathbb I}) \to {\cal C}$, where ${\cal G}({\mathbb I})$ is the category obtained from $\mathbb I$ by applying the Grothendieck construction. Given a Grothendieck topology $J$ on $\cal C$, there is a smallest topology $J^{\textup{Gir}}_{\mathbb I}$ on ${\cal G}({\mathbb I})$ which makes $\pi_{\mathbb I}$ a comorphism of sites to $({\cal C}, J)$; this topology, which we call the \emph{Giraud topology}, has as covering sieves those which contain cartesian lifts of $J$-covering families in ${\cal C}$.

Let $f:{\cal F}\to {\cal E}$ be a geometric morphism. There are three pseudofunctors naturally associated with it:

\begin{itemize}
\item We define 
\[
t_{f}:{\cal F}^{\textup{op}}\to \textup{Cat}
\]	
as the functor sending any object $F$ of $\cal F$ to the category $(F \downarrow f^{\ast})$ and any arrow $v:F\to F'$ to the functor $(F' \downarrow f^{\ast})\to (F \downarrow f^{\ast})$ induced by composition with $v$.

\item We define  
\[
r_{f}:{\cal E}^{\textup{op}}\to \textup{Cat}
\]
as the pseudofunctor sending an object $E$ of $\cal E$ to ${\cal F}\slash f^{\ast}(E)$ and an arrow $u:E\to E'$ to the pullback functor $(f^{\ast}(u))^{\ast}:{\cal F}\slash f^{\ast}(E') \to {\cal F}\slash f^{\ast}(E)$.

\item We define 
\[
s_{f}:{\cal E} \to \textup{Cat}
\]
as the functor sending an object $E$ of $\cal E$ to ${\cal F}\slash f^{\ast}(E)$ and an arrow $u:E\to E'$ to the functor $\Sigma_{f^{\ast}(u)}:{\cal F}\slash f^{\ast}(E) \to {\cal F}\slash f^{\ast}(E')$ induced by composition with $f^{\ast}(u)$ (which is left adjoint to the pullback functor $(f^{\ast}(u))^{\ast}:{\cal F}\slash f^{\ast}(E') \to {\cal F}\slash f^{\ast}(E)$).
\end{itemize}

By the adjunction between $\Sigma_{f^{\ast}(u)}$ and $(f^{\ast}(u))^{\ast}$ (for any arrow $u$ in $\cal E$), the fibration to $\cal E$ associated with $r_{f}$ coincides with the opfibration associated with $s_{f}$; in particular, this functor is both a fibration and an opfibration. This fibration is a stack for the canonical topology on $\cal E$ by the results in \cite{giraud}. Note that the domain of this fibration also admits a canonical functor to $\cal F$, which is precisely the fibration associated with $t_{f}$. 
 
\begin{proposition}
	$t_{f}$ is a (split) stack for the canonical topology on $\cal F$.
\end{proposition}

\begin{proof}
Let $\{f_{i}:F_{i} \to F \mid i\in I\}$ be an epimorphic family in $\cal F$ and 
	\[
	\{ A_i=(F_{i}, E_{i}, \alpha_i:F_{i} \to f^{\ast}(E_{i})) \mid i\in I\}, \{f_{ij}: {\pi_{i}}^{\ast}(A_{i}) \xrightarrow{\raisebox{-0.7ex}[0ex][0ex]{$\sim$}} {\pi'_j}^{\ast}(A_{j}) \mid (i, j)\in I\times I\},
	\]
	where $\pi_{i}$ and $\pi'_{j}$ are defined, for each $(i, j)\in I\times I$, by the pullback square
\[\begin{tikzcd}
	{F_{i}\times_{F}F_{j}} & {F_i} \\
	{F_{j}} & F,
	\arrow["{\pi_{i}}", from=1-1, to=1-2]
	\arrow["{\pi'_{j}}"', from=1-1, to=2-1]
	\arrow["{f_j}"', from=2-1, to=2-2]
	\arrow["{f_{i}}", from=1-2, to=2-2]
\end{tikzcd}\]
be a collection of descent data indexed by it. For any $(i, j)\in I$, 
\[
\pi_{i}^{\ast}(A_{i})=(F_{i}\times_{F}F_{j}, E_{i}, \alpha_i\circ \pi_{i}:F_{i}\times_{F}F_{j} \to f^{\ast}(E_{i})),
\]
\[
{\pi'_{j}}^{\ast}(A_{j})=(F_{i}\times_{F}F_{j}, E_{j}, \alpha_j\circ \pi'_{j}:F_{i}\times_{F}F_{j} \to f^{\ast}(E_{j})).
\]
So the isomorphism $f_{ij}$ actually identifies with an isomorphism 
\[
f_{ij}:E_{i}\xrightarrow{\raisebox{-0.7ex}[0ex][0ex]{$\sim$}} E_{j}
\]
in $\cal E$ such that the following triangle commutes:
\[\begin{tikzcd}
	{F_{i}\times_{F}F_{j}} && {f^{\ast}(E_i)} \\
	&& {f^{\ast}(E_j)}
	\arrow["{f^{\ast}(f_{ij})}", from=1-3, to=2-3]
	\arrow["{\alpha_j\circ \pi'_{j}}"', from=1-1, to=2-3]
	\arrow["{\alpha_i \circ \pi_i}", from=1-1, to=1-3]
\end{tikzcd}\]

Note that the fact that $t_{f}$ is split easily implies that
\begin{itemize}
    \item for any $i\in I$, $f_{ii}=1_{A_i}$,
    
    \item for any $i, j, k\in I$, $f_{jk}\circ f_{ij}=f_{ik}$. 
\end{itemize}

Let $E$ be the colimit of the diagram in $\cal E$ having the $E_{i}$'s as vertices and the $f_{ij}$'s as edges between them (with the above relations); the above identities actually imply that, for each $i\in I$, $E\cong E_{i}$. The commutativity of the above triangles ensures that we have a cocone from the $F_{i}$'s to $f^{\ast}(E)$ whose legs are given by the composites of the arrows $\alpha_i$ with the image under $f^{\ast}$ of the corresponding canonical colimit arrow $E_{i}\to E$. Since the representable $\textup{Hom}_{\cal F}(-, f^{\ast}(E))$ is a sheaf for the canonical topology on $\cal F$, it follows that there is a unique arrow $F\to f^{\ast}(E)$ which restricts on the $F_i$'s to the legs of this cocone, and which therefore provides the required \ac amalgamation' for our descent data. The uniqueness of the amalgamation (up to isomorphism) also follows at once from the sheaf property.       
\end{proof}

\begin{corollary}
	Let $M$ be a model of a geometric theory $\mathbb T$ in a Grothendieck topos $\cal E$. Then the functor 
	\[
	t_{f_{M}}:{\cal E}^{\textup{op}} \to \textup{Cat}
	\] 
	associated with the geometric morphism $f_{M}:{\cal E}\to \Sh({\cal C}_{\mathbb T}, J_{\mathbb T})$ corresponding to $M$ via the universal property of the classifying topos for $\mathbb T$ is a stack for the canonical topology on $\cal E$.
	
	In particular, if $({\cal C}, J)$ is a site of definition for $\mathcal E$ then the functor
	\[  \begin{array}{rcl}
	\mathcal{C}^{\textup{op}} & \stackrel{\mathbb{M}}{\longrightarrow} & \textup{Cat} \\
	c  & \longmapsto & (c \downarrow F_{M}), \\
	c_1 \stackrel{u}{\rightarrow} c_2 & \longmapsto & (c_2 \downarrow F_{M}) \stackrel{u^*}{\rightarrow} (c_1 \downarrow F_{M})
	\end{array} \]
	where $u^* : (c_2 \downarrow F_{M}) {\rightarrow} (c_1 \downarrow F_{M})$ is the pre-composition functor sending any generalized element $ a : {c_2} \rightarrow \llbracket \vec{x}^{\vec{A}}.  \phi \rrbracket_M $ to $a\circ u : {c_1} \rightarrow \llbracket \vec{x}^{\vec{A}}.  \phi \rrbracket_M $
	is a stack for the topology $J$, where $F_{M}$ is the functor ${\cal C}_{\mathbb T}\to {\cal E}$ taking the interpretations of formulae in the model $M$.
\end{corollary}\qed
 
The category $(1_{\cal F}\downarrow f^{\ast})={\cal G}(r_{f})={\cal G}(t_{f})$ has as objects the triplets $(F, E, \alpha:F\to f^{\ast}(E))$ (where $E\in {\cal E}$, $F\in {\cal F}$ and $\alpha$ is an arrow in $\cal F$) and as arrows $(F, E, \alpha:F\to f^{\ast}(E)) \to (F', E', \alpha':F'\to f^{\ast}(E'))$ the pairs of arrows $(v:F\to F', u:E\to E')$ in $\cal F$ and $\cal E$ such that $f^{\ast}(u)\circ \alpha=\alpha'\circ v$; the functors $\pi_{r_{f}}$ and $\pi_{t_{f}}$ are respectively the canonical projection functors from this category to $\cal E$ and $\cal F$. Since via these functors the category ${\cal G}(r_{f})={\cal G}(t_{f})$ is fibered both over $\cal E$ and over $\cal F$, it is natural to consider the smallest Grothendieck topology on it which makes $\pi_{r_{f}}$ and $\pi_{t_{f}}$ comorphisms of sites when $\cal E$ and $\cal F$ are endowed with their canonical topologies, in other words the join of the Giraud topologies on it induced by the canonical topologies on $\cal E$ and $\cal F$. As we shall see, this topology will play a crucial role in connection with our construction of the over-topos. To this end, we more generally describe, for any category $\cal C$ and basis $\cal B$ for a Grothendieck topology on $\cal C$ such that $({\cal C}, J_{\cal B})$ is a site of definition for $\cal E$, a basis for the Grothendieck topology on ${\cal G}(r_{f}')$, where $r_{f}'$ is the \ac restriction' of $r_{f}$ to $\cal C$ (that is, the composite of $r_{f}$ with the opposite of the canonical functor ${\cal C}\to {\cal E}$), which is generated by the Giraud topology induced by the canonical topology on $\cal F$ and by the Giraud topology induced by $J_{\cal B}$. This result will be applied subsequently to the syntactic site $({\cal C}_{\mathbb T}, J_{\mathbb T})$ of definition of a geometric theory $\mathbb T$ inside its classifying topos.

\begin{definition}\label{defliftedtopology}
Given a Grothendieck topology $J$ on $\cal C$ (resp. a basis $\cal B$ for a Grothendieck topology on $\cal C$) such that $({\cal C}, J)$ (resp. $({\cal C}, J_{\cal B})$) is a site of definition for the topos $\cal E$, we shall call the Grothendieck topology on ${\cal G}(r_{f}')$ generated by the Giraud topology induced by the canonical topology on $\cal F$ and by the Giraud topology induced by $J$ (resp. by $J_{\cal B}$) the \emph{$(f, J)$-lifted topology} (resp. the \emph{$(f, {\cal B})$-lifted topology}) and shall denote it by $L_{(f, J)}$ (resp. $L_{(f, {\cal B})}$). 
\end{definition}

\begin{theorem}\label{thmliftedtopology}
Let $f:{\cal F}\to {\cal E}$ be a geometric morphism and $\cal B$ be a basis for a Grothendieck topology on $\cal C$ such that $({\cal C}, J_{\cal B})$ is a site of definition for $\cal E$. Then, with the above notation, the $(f, {\cal B})$-lifted Grothendieck topology on ${\cal G}(r_{f}')$ has as a basis the collection of multicomposites of a family of cartesian lifts (with respect to $r_{f}'$) of arrows in a family of $\cal B$ with $J^{\textup{Gir}}_{\cal F}$-covering families, that is, the collection of families  
\[ \begin{tikzcd}[column sep=large]
     ((d_{ij}, (c_i, b_{ij}))  \arrow[]{rr}{ (u_{ij}, \, (\xi_i, \widetilde{b_{ij}}))} && (F, (c, a)))_{i \in I, j \in J_i}
    \end{tikzcd}  \]
    where $(\xi_{i}: c_i \rightarrow  c)_{i \in I}$ is a family in ${\cal B}(c)$ and the families \[ (\widetilde{b_{ij}} : d_{ij} \rightarrow f^{\ast}(c_i)\times_{f^{\ast}(c)} F)_{j \in J_i}\] are epimorphic in $\cal E$ for each $i\in I$:
    
    \[ 
\begin{tikzcd}
d_{ij} \arrow[rrd, "u_{ij}", bend left=20] \arrow[rdd, "b_{ij}"', bend right=20] \arrow[rd, dashed, "\widetilde{b_{ij}}"] &                                                                                             &                                      \\
                                                                                              & f^{\ast}(c_i)\times_{f^{\ast}(c)} F \arrow[r, "\pi_{i}"] \arrow[d, "a_i"] \arrow[rd, "\lrcorner", phantom, near start] & F \arrow[d, "a"]                     \\
                                                                                              & f^{\ast}(c_i) \arrow[r, "f^{\ast}(\xi_i)"]    & f^{\ast}(c)
\end{tikzcd} \]
    
\end{theorem}

\begin{proof}
Note that the collection of arrows
\[ \begin{tikzcd}[column sep=large]
     \big{(}(d_{ij}, (c_i, b_{ij}))  \arrow[]{rr}{ (u_{ij},\, (\xi_i, \widetilde{b_{ij}}))} && (F, (c, a))\big{)}_{i \in I, j \in J_i}
    \end{tikzcd}  \]
is the multicomposite of the family 
\[
\big{(}(\pi_{i},\, (\xi_i, 1)):(f^{\ast}(c_i) \times_{f^{\ast}(c)} F, (c_i, a_i)) \to (F, (c, a))\big{)}_{i\in I},
\]
each of whose arrows is cartesian with respect to the fibration $r_{f}$ to $\cal E$, 
with the families 
\[
\big{(} (\widetilde{b_{ij}},\, (1_{c_{i}}, \widetilde{b_{ij}})) :(d_{ij}, (c_i, b_{ij})) \to (f^{\ast}(c_i) \times_{f^{\ast}(c)} F, (c_i, a_i))  \big{)}_{j\in J_{i}}
\]
(for $i\in I$), each of whose arrows is cartesian with respect to the fibration $t_{f}$ to $\cal F$ (and vertical with respect to the fibration $r_{f}$).

The first condition in the definition of basis is clearly satisfied, since $\cal B$ is a basis. The second condition follows from the stability under pullback of epimorphic families in $\cal F$ as well as of families in $\cal B$, in light of the compatibility of multicomposition with respect to pullback. It remains to show that the collection of families specified in the theorem is closed with respect to multicomposition.

Let 
\[ \begin{tikzcd}[column sep=large]
     \big{(}(d_{ij}, (c_i, b_{ij}))  \arrow[]{rr}{ (u_{ij},\, (\xi_i, \widetilde{b_{ij}}))} && (F, (c, a))\big{)}_{i \in I,\, j \in J_i}
    \end{tikzcd} 
\]
and, for each $i \in I, j \in J_i$,
\[ \begin{tikzcd}[column sep=large]
     \big{(}(d^{ij}_{kl}, (c^{ij}_k, b^{ij}_{kl}))  \arrow[]{rr}{ (u^{ij}_{kl},\, (\xi^{ij}_k, \widetilde{b^{ij}_{kl}}))} && ((d_{ij}, (c_i, b_{ij}))\big{)}_{k \in K^{ij},\, l \in L^{ij}_k}
    \end{tikzcd} 
\]

    \[ 
\begin{tikzcd}
d^{ij}_{kl} \arrow[rrd, "u^{ij}_{kl}", bend left=20] \arrow[rdd, "b^{ij}_{kl}"', bend right=20] \arrow[rd, dashed, "\widetilde{b^{ij}_{kl}}"] &                                                                                             &                                      \\
                                                                                              & f^{\ast}(c^{ij}_{k})\times_{f^{\ast}(c_i)} d_{ij} \arrow[r] \arrow[d] \arrow[rd, "\lrcorner", phantom, near start] & d_{ij} \arrow[d, "b_{ij}"]                     \\
                                                                                              & f^{\ast}(c^{ij}_{k}) \arrow[r, "f^{\ast}(\xi^{ij}_{k})"]    & f^{\ast}(c_i)
\end{tikzcd} \]
be families satisfying the conditions in the theorem.

We want to prove that their multicomposite also satisfies these conditions.

As ${\cal B}$ is a basis for a Grothendieck topology, the family 
\[  \begin{tikzcd}
\big{(} c^{ij}_k  \arrow[rr, "\xi_i\circ \xi^{ij}_k"] & & c \big{)}_{i \in I, j \in J_i, k \in K^{ij}}
\end{tikzcd}   \]
is in $\cal B$.

Consider the following pullback diagrams:

\[\begin{tikzcd}
	{f^{\ast}(c^{ij}_{k})\times_{f^{\ast}(c)} F} && {f^{\ast}(c_i)\times_{f^{\ast}(c)} F} && F \\
	{f^{\ast}(c^{ij}_{k})} && {f^{\ast}(c_i)} && {f^{\ast}(c)}
	\arrow["{f^{\ast}(\xi^{ij}_{k})}", from=2-1, to=2-3]
	\arrow["{f^{\ast}(\xi_i)}", from=2-3, to=2-5]
	\arrow["a", from=1-5, to=2-5]
	\arrow[from=1-3, to=1-5]
	\arrow[from=1-3, to=2-3]
	\arrow[from=1-1, to=1-3]
	\arrow[from=1-1, to=2-1]
	\arrow["\lrcorner"{anchor=center, pos=0.125}, draw=none, from=1-1, to=2-3]
	\arrow["\lrcorner"{anchor=center, pos=0.125}, draw=none, from=1-3, to=2-5]
	\arrow["{f^{\ast}(\xi_{i}\circ \xi^{ij}_{k})}"', bend right=15, from=2-1, to=2-5]
\end{tikzcd}\]
For each $i\in I$ and $j\in J_{i}$, in the diagram
\[ 
\begin{tikzcd}[column sep=small]
\underset{k \in K^{ij} \atop l \in L^{ij}_k}{\coprod} d^{ij}_{kl} \arrow[rr, two heads] \arrow[rrrdd, "\langle b^{ij}_{kl} \rangle_{k \in K^{ij} \atop l \in L^{ij}_k}"', bend right=20] \arrow[rrrr, "\langle u^{ij}_{kl} \rangle_{k \in K^{ij} \atop l \in L^{ij}_k}", bend left=20] &  & f^{\ast}(c^{ij}_{k}) \times_{f^{\ast}(c)} d_{ij} \arrow[rd] \arrow[rr] \arrow[rdd, bend right=20] \arrow[rrd, "\lrcorner", phantom, very near start] &                                                                                                                & d_{ij} \arrow[rd] \arrow[rdd, bend right=20] &                                                       \\
                                               &  &                                                                                                                                                                     & f^{\ast}(c^{ij}_{k})\times_{f^{\ast}(c)} F  \arrow[rr, crossing over] \arrow[d] \arrow[rrd, "\lrcorner", phantom, very near start] & {}                                                  & f^{\ast}(c_i)\times_{f^{\ast}(c)} F \arrow[d]    \\
                                           &  &                                                                                   & f^{\ast}(c^{ij}_{k}) \arrow[rr, "f^{\ast}(\xi^{ij}_{k})"']                              &                                                     & f^{\ast}(c_i)
\end{tikzcd} \]
both the front lower and back squares are pullbacks, hence so is the top square. Then in the diagram 
\[ 
\begin{tikzcd}[column sep=tiny]
\underset{i \in I \atop j \in J_i}{\coprod}\underset{k \in K^{ij} \atop l \in L^{ij}_k}{\coprod} d^{ij}_{kl} \arrow[rr, two heads] \arrow[rrrdd, "\langle\langle b^{ij}_{kl} \rangle_{k \in K^{ij} \atop l \in L^{ij}_k}\rangle_{i \in I \atop j \in J_i}"', bend right=20] \arrow[rrrr, "\langle\langle u^{ij}_{kl} \rangle_{k \in K^{ij} \atop l \in L^{ij}_k}\rangle_{i \in I \atop j \in J_i}", bend left=20] &  & \underset{i \in I \atop j \in J_i}{\coprod} \underset{k \in K^{ij}}{\coprod}(f^{\ast}(c^{ij}_{k}) \times_{f^{\ast}(c)}d_{ij})
 \arrow[rd] \arrow[rr] \arrow[rdd, bend right=20] \arrow[rrd, "\lrcorner", phantom, very near start] &                                                                                                                & \underset{i \in I \atop j \in J_i}{\coprod} d_{ij} \arrow[rd, two heads, "\langle \widetilde{b_{ij}}\rangle_{i \in I \atop j \in J_i}"] \arrow[rdd, bend right=20] &                                                       \\                                               &  &                                                                                                                                                                                                  & \underset{i \in I \atop j \in J_i}{\coprod} \underset{k \in K^{ij}}{\coprod} (f^{\ast}(c^{ij}_{k}) \times_{f^{\ast}(c)} F)
                                          \arrow[rr, crossing over] \arrow[d] \arrow[rrd, "\lrcorner", phantom, very near start] & {}                                                                                           & \underset{i \in I}{\coprod}(f^{\ast}(c_i)\times_{f^{\ast}(c)} F)   \arrow[d]    \\
                                           &  &                                                                                              & \underset{i \in I \atop j \in J_i}{\coprod} \underset{k \in K^{ij}}{\coprod} f^{\ast}(c^{ij}_{k}) \arrow[rr, " \underset{i \in I}{\coprod} \langle f^{\ast}(\xi^{ij}_{k}) \rangle_ {j \in J_i \atop k \in K^{ij}}"']                              &                                                                                              & \underset{i \in I }{\coprod}f^{\ast}(c_i)
\end{tikzcd} \]
the upper square is a pullback, whence the left-hand arrow in it is an epimorphism (as it is the pullback of the epimorphism $\langle \widetilde{b_{ij}}\rangle_{i \in I \atop j \in J_i}$). Therefore the arrow $$\langle\langle \widetilde{b^{ij}_{kl}} \rangle_{k \in K^{ij} \atop l \in L^{ij}_k}\rangle_{i \in I \atop j \in J_i}:\underset{i \in I \atop j \in J_i}{\coprod}\underset{k \in K^{ij} \atop l \in L^{ij}_k}{\coprod} d^{ij}_{kl} \to \underset{i \in I \atop j \in J_i}{\coprod} \underset{k \in K^{ij}}{\coprod} (f^{\ast}(c^{ij}_{k}) \times_{f^{\ast}(c)} F)$$ is also an epimorphism, as it is the composite of two epimorphisms. But this is precisely the arrow corresponding to our multicomposite family, whence we can conclude that the latter satisfies the conditions of the theorem, as required (as coproducts are disjoint in a topos, a coproduct of arrows is an epimorphism if and only if each of the arrows are). 
\end{proof}

We now apply the theorem to the geometric morphism to the classifying topos of a geometric theory $\mathbb T$, represented as the topos of sheaves on its geometric syntactic site, induced by a model of $\mathbb T$:

\begin{corollary}\label{corantecedenttopology}
   Let $M$ be a model of a geometric theory $\mathbb T$ in a Grothendieck topos $\cal E$ and $f_{M}:{\cal E}\to \Sh({\cal C}_{\mathbb T}, J_{\mathbb T})$ the geometric morphism corresponding to it via the universal property of the classifying topos for $\mathbb T$. 
   Then we have: 
   
   \begin{enumerate}
       \item[(i)] The $(f_{M}, J_{\mathbb T})$-lifted topology $L_{(f_{M}, J_{\mathbb T})}$ on $(1_{\cal E} \downarrow F_{M})$ has as a basis the collection of families
    \[ \begin{tikzcd}[column sep=large]
     \big{(}(d_{ij}\, (\{ \vec{x}_i^{\vec{A}_i}. \phi_i \} \, , \llbracket \theta_i \rrbracket_{M}(a))) \arrow[]{rr}{ (u_{ij} , \, ([\theta_i ]_\mathbb{T}, \widetilde{b_{ij}}))} && (e, (\{ \vec{x}^{\vec{A}}. \phi \} \, , a)\big{)}_{i \in I, j \in J_i}
    \end{tikzcd}  \]
    where $e\in {\cal E}$, $([\theta_i ]_\mathbb{T} : \{ \vec{x}_i^{\vec{A}_i}. \phi_i \} \rightarrow  \{ \vec{x}^{\vec{A}}. \phi \})_{i \in I}$ is a family in ${\cal B}_{\mathbb T}(\{ \vec{x}^{\vec{A}}. \phi \})$ and the families \[ \big{(}\widetilde{b_{ij}} : d_{ij} \rightarrow \llbracket \theta_i \rrbracket_M ^{-1}(a)\big{)}_{j \in J_i}\] are epimorphic in $\cal E$ for each $i\in I$:
    \[ 
\begin{tikzcd}
d_{ij} \arrow[rrd, "u_{ij}", bend left=20] \arrow[rdd, "b_{ij}"', bend right=20] \arrow[rd, dashed, "\widetilde{b_{ij}}"] &                                                                                             &                                      \\
                                                                                              & \llbracket \theta_i \rrbracket_M^{-1}(a) \arrow[r, "\pi_{i}"] \arrow[d, "\llbracket \theta_i \rrbracket_{M}(a)"] \arrow[rd, "\lrcorner", phantom, very near start] & e \arrow[d, "a"]                     \\
                                                                                              & \llbracket \vec{x}_i^{\vec{A}_i}. \phi_i \rrbracket_M \arrow[r, "\llbracket \theta_i \rrbracket_M"]    & \llbracket \vec{x}^{\vec{A}}. \phi \rrbracket_M
\end{tikzcd} \]
    
    \item[(ii)] For any separating set ${\cal C}$ for $\cal E$, the canonical functor ${\cal C}\to {\cal E}$ induces a $L_{(f_{M}, J_{\mathbb T})}$-dense functor $(i_{\cal C} \downarrow F_{M}) \to (1_{\cal E} \downarrow F_{M})$, where $i_{\cal C}$ is the canonical embedding of $\cal C$ in $\cal E$, and the Grothendieck topology induced by $L_{(f_{M}, J_{\mathbb T})}$ on the category $(i_{\cal C} \downarrow F_{M})$ has as a basis the collection of families
    \[ \begin{tikzcd}[column sep=large]
     \big{(}(d_{ij},\, (\{ \vec{x}_i^{\vec{A}_i}. \phi_i \} \, , \llbracket \theta_i \rrbracket_M(a))) \arrow[]{rr}{ (u_{ij} , \, ([\theta_i ]_\mathbb{T}, \widetilde{b_{ij}}))} && (c, (\{ \vec{x}^{\vec{A}}. \phi \} \, , a)\big{)}_{i \in I, j \in J_i}
    \end{tikzcd}  \]
    where $c, d_{ij}\in {\cal C}$ for each $i$ and $j$, $([\theta_i ]_\mathbb{T} : \{ \vec{x}_i^{\vec{A}_i}. \phi_i \} \rightarrow  \{ \vec{x}^{\vec{A}}. \phi \})_{i \in I}$ is a family in ${\cal B}_{\mathbb T}(\{ \vec{x}^{\vec{A}}. \phi \})$ and the families \[ \big{(}\widetilde{b_{ij}} : d_{ij} \rightarrow \llbracket \theta_i \rrbracket_M^{-1}(a)\big{)}_{j \in J_i}\] are epimorphic in $\cal E$ for each $i\in I$:
    \[ 
\begin{tikzcd}
d_{ij} \arrow[rrd, "u_{ij}", bend left=20] \arrow[rdd, "b_{ij}"', bend right=20] \arrow[rd, dashed, "\widetilde{b_{ij}}"] &                                                                                             &                                      \\
                                                                                              & \llbracket \theta_i \rrbracket_M^{-1}(a) \arrow[r, "\pi_{i}"] \arrow[d, "\llbracket \theta_i \rrbracket_M(a)"] \arrow[rd, "\lrcorner", phantom, very near start] & c \arrow[d, "a"]                     \\
                                                                                              & \llbracket \vec{x}_i^{\vec{A}_i}. \phi_i \rrbracket_M \arrow[r, "\llbracket \theta_i \rrbracket_M"]    & \llbracket \vec{x}^{\vec{A}}. \phi \rrbracket_M
\end{tikzcd} \]

    \item[(iii)] If $\cal E$ is the topos $\Set$ of sets and $1:\{\ast\}\to {\cal E}$ is the functor from the one-object and one-arrow category $\{\ast\}$ to $\cal E$ picking the terminal object of $\Set$, the embedding $(\int M)=(1_{\{\ast\}} \downarrow F_{M}) \to (1_{\cal E}\downarrow F_{M})$ induced by the functor $1$ is $L_{(f_{M}, J_{\mathbb T})}$-dense and the topology induced by $L_{(f_{M}, J_{\mathbb T})}$ on $(\int M)$ is the antecedent topology of Definition \ref{antecedents topology}. 

   \end{enumerate}
\end{corollary}

\begin{proof}
    (i) This is the particular case of Theorem \ref{thmliftedtopology} by taking $\cal E$ to be the classifying topos of $\mathbb T$ and $({\cal C}, J)$ the geometric syntactic site of $\mathbb T$.
    
    (ii) This easily follows from the definition of induced topology on a full dense subcategory.
    
    (iii) This is the particular case of (ii) when $\cal E$ is $\Set$ and ${\cal C}$ is the category $\{\ast\}$, regarded as a separating set for $\cal E$ through the embedding $1:\{\ast\}\to {\cal E}$.
\end{proof}

We now fix a Grothendieck topos $ \mathcal{E}$ and a small subcanonical cartesian site $(\mathcal{C},J)$ of definition for $\mathcal{E}$.

\begin{definition}
A \emph{cartesian} stack is a stack $ \mathbb{M} $ on $(\mathcal{C},J)$ such that any $ \mathbb{M}(c)$ is a cartesian category, and such that any transition functor $ \mathbb{M}(u)$, for $u$ a morphism in $\mathcal{C}$, is cartesian. A morphism of cartesian stacks is a morphism of stacks $ \alpha : \mathbb{M}_1 \rightarrow \mathbb{M}_2$ such that in any $ c$, $\alpha_c :\mathbb{M}_1(c) \rightarrow \mathbb{M}_2(c) $ is cartesian. 

We shall denote by $\St_{\textup{cart}}({\cal C}, J)$ the category of stacks on a site $({\cal C}, J)$ and morphisms of cartesian stacks between them.
\end{definition}

The following lemma, whose proof is straightforward, expresses cartesian lifts in cartesian fibrations as pullback squares, and  will serve for describing one half of the correspondence provided by Giraud's construction of the classifying topos of a stack recalled below:

\begin{lemma}\label{liftterminal}
Let $ \mathbb{M} $ be a cartesian stack, with $ 1_\mathbb{M}(c)$ the terminal object of the fiber of ${\mathbb M}$ at $c$. Then, in the Grothendieck fibration $\pi_\mathbb{M} : \int \mathbb{M} \rightarrow \mathcal{C}$, for any arrow $ u:  c_1 \rightarrow c_2$ and any object $ a$ in $ \mathbb{M}(c_2)$, the following square is a pullback:
\[
\begin{tikzcd}
{(c_1, \mathbb{M}(u)(a))} \arrow[d, "{(1_{c_1}, !_{\mathbb{M}(u)(a)})}"'] \arrow[r, "{(u, 1_{\mathbb{M}(u)(a)})}"] \arrow[rd, "\lrcorner", very near start, phantom] & {(c_2, a)} \arrow[d, "{(1_{c_2}, !_{a})}"] \\
{(c_1, 1_{\mathbb{M}(c_1)})} \arrow[r, "{(u, 1_{1_{\mathbb{M}(c_1)}})}"]                                                                            & {(c_2, 1_{\mathbb{M}(c_2)})}              
\end{tikzcd} \]
\end{lemma}\qed

\begin{proposition}\label{propcartesianstructure}
For any cartesian stack on a site $ \mathbb{M}$ on $(\mathcal{C},J)$ whose underlying category $\cal C$ is cartesian, $ \int \mathbb{M}$ is a cartesian category. 
\end{proposition}

\begin{proof}
We must prove that $ \int \mathbb{M}$ is cartesian. As we shall see, the considered property holds globally on $\int \mathbb{M}$ out of holding in a specific fiber. Let $ (c_i, a_i)_{i \in I}$ a finite diagram in $ \int \mathbb{M}$; then, as $ \mathcal{C}$ is cartesian, we can compute the limit $p_{i}:\lim_{i \in I} c_i \to c_i$ in $ \mathcal{C}$. Moreover, $ \mathbb{M}(\lim_{i \in I} c_i)$ is cartesian, so the finite limit $ \lim_{i \in I} \mathbb{M}(p_i)(a_i)$ also exists in $ \mathbb{M}(\lim_{i \in I} c_i)$, providing a cone \[((p_i, \pi_i) : (\underset{i \in I}{\lim}   \; c_i,
\underset{i \in I}{\lim}  \; \mathbb{M}(p_i)(a_i)) \rightarrow (c_i, a_i))_{i \in I}  \] in $ \int \mathbb{M}$, where $ \pi_i : \lim_{i \in I} \mathbb{M}(p_i)(a_i) \rightarrow \mathbb{M}(p_i)(a_i)$ is the projection in the fiber. Now for any other cone $ ((v_i, q_i) : (c,a) \rightarrow (c_i, a_i))_{i \in I}$, with $ q_i : a \rightarrow \mathbb{M}(v_i)(a_i)$, there exists a unique arrow $ w :c \rightarrow \lim_{i \in I} c_i$ by the universal property of the limit in $ \mathcal{C}$; but, as the transition functors $ \mathbb{M}(p_i)$ are cartesian, we have 
\[ \mathbb{M}(w)(\underset{i \in I}{\lim}  \;\mathbb{M}(p_i)(a_i)) \simeq \underset{i \in I}{\lim} \; \mathbb{M}(w) \mathbb{M}(p_i)(a_i) \simeq \underset{i \in I}{\lim} \; \mathbb{M}(v_i)(a_i)  \]
inducing a unique arrow $ f : a \rightarrow  \mathbb{M}(w)({\lim}_{i \in I} \mathbb{M}(p_i)(a_i)) $, so that there is a unique factorization in $ \int \mathbb{M}$
\[ 
\begin{tikzcd}
{(c,a)} \arrow[rr, "{(w,f)}", dashed] \arrow[rd, "{(v_i, q_i)}"'] &              & {(\underset{i \in I}{\lim} \; c_i, \underset{i \in I}{\lim} \; \mathbb{M}(p_i)(a_i))} \arrow[ld, "{(p_i, \pi_i)}"] \\
                                                                  & {(c_i, a_i)} &                             
\end{tikzcd} \]
as desired.
\end{proof}

Let us recall from \cite{giraud} the following fundamental classification result:

\begin{theorem}\label{thmGiraud}
Let $\mathbb M$ be a cartesian stack on a cartesian site $({\cal C}, J)$ of definition for a topos $\cal E$, and $J^{\textup{Gir}}_\mathbb{M}$ be the Giraud topology induced by $J$. Then the sheaf topos $\mathcal{E}[\mathbb{M}]=\Sh(\int \mathbb{M}, J^{\textup{Gir}}_\mathbb{M})$, with its canonical morphism  $p_{\mathbb M}:\mathcal{E}[\mathbb{M}] \to {\cal E}$ induced by the comorphism of sites $\pi_{{\mathbb M}}:(\textstyle{\int} \mathbb{M}, J^{\textup{Gir}}_\mathbb{M}) \to ({\cal C}, J)$, is the classifier of $\mathbb{M}$, in the sense that the following universal property holds: for any geometric morphism $g:{\cal G}\to {\cal E}$,
\[ \textup{\bf Geom}_\mathcal{E}[g, p_{\mathbb M}] \simeq \St_{\textup{cart}}(\mathcal{C}, J)[\mathbb{M}, \mathcal{G}/g^*]. \]
\end{theorem}\qed

Given a triangle  \[ 
\begin{tikzcd}
\mathcal{G} \arrow[rd, "g"'] \arrow[rr, "f"] &             & {\mathcal{E}[\mathbb{M}]} \arrow[ld, "p_{\mathbb M}"] \\
                                             & \mathcal{E} &                                                     
\end{tikzcd} \]
of geometric morphisms, the cartesian morphism of stacks $\alpha:{\mathbb M}\to {\cal G}\slash g^{\ast}$ associated with $g$ as in the theorem can be described as follows.

We have $ g^* \cong f^*p_{\mathbb M}^*$. The restriction to $\cal C$ of the functor $p_{\mathbb M}^*$ is just the terminal section $1_{\mathbb{M}(-)}$, so we have a commutative (up to isomorphism) triangle 
\[ 
\begin{tikzcd}
\int \mathbb{M} \arrow[rr, "f^*"] &                                                                       & \mathcal{G} \\
                                         & \mathcal{C} \arrow[lu, "1_{\mathbb{M}(-)}"] \arrow[ru, "g^*"'] &            
\end{tikzcd} \]

Now, for any $ (c,a) $ in $ \int \mathbb{M}$, there is a unique arrow $ !_a : a \rightarrow 1_{\mathbb{M}(c)} $ in the fiber $ \mathbb{M}(c)$, which is sent by $f^{\ast}$ to an arrow $ f^*(1_c,!_a) : f^*(c,a) \rightarrow f^*(c,1_{\mathbb{M}(c)}) \simeq g^*(c)$. This yields, for each $c\in {\cal C}$, a  functor 
\[ \begin{array}{rcl}
    \alpha_{c}: \mathbb{M}(c) & \rightarrow & \mathcal{G}/g^*(c) \\
    a & \mapsto & f^*(1_c,!_a) : f^*(c,a) \rightarrow g^*(c)
\end{array} \]
The functors $ (\alpha_c : \mathbb{M}(c) \rightarrow \mathcal{G}/g^*(c))_{c \in \mathcal{C}}$ actually define a natural transformation, that is, all the squares of the following form commute up to isomorphism:
\[ 
\begin{tikzcd}
\mathbb{M}(c_2) \arrow[r, "\alpha_{c_2}"] \arrow[d, "\mathbb{M}(u)"'] & \mathcal{G}/g^*(c_2) \arrow[d, "(g^*(u))^*"] \\
\mathbb{M}(c_1) \arrow[r, "\alpha_{c_1}"]                             & \mathcal{G}/g^*(c_1)                           
\end{tikzcd} \]
Indeed, by Lemma \ref{liftterminal} the following square is a pullback in $ \int \mathbb{M}$:
\[ \begin{tikzcd}
{(c_1, \mathbb{M}(u)(a))} \arrow[d, "{(1_{c_1}, !_{\mathbb{M}(u)(a)})}"'] \arrow[r, "{(u, 1_{\mathbb{M}(u)(a)})}"] \arrow[rd, "\lrcorner", very near start, phantom] & {(c_2, a)} \arrow[d, "{(1_{c_2}, !_{a})}"] \\
{(c_1, 1_{\mathbb{M}(c_1)})} \arrow[r, "{(u, 1_{1_{\mathbb{M}(c_1)}})}"]                                                                            & {(c_2, 1_{\mathbb{M}(c_2)})}              
\end{tikzcd}  \]
It thus follows that, the functor $ f^*$ being cartesian, the following square is also a pullback:
\[ 
\begin{tikzcd}[column sep=huge]
{f^*(c_1,\mathbb{M}(u)(a)))} \arrow[d, "\alpha_{c_1}(\mathbb{M}(u)(a))"'] \arrow[r, "{f^*(u,1_{\mathbb{M}(u)(a)})}"] \arrow[rd, "\lrcorner", very near start, phantom] & {f^*(c_2,a)} \arrow[d, "\alpha_{c_2}(a)"] \\
g^*(c_1) \arrow[r, "g^*u"]                                                                                                                           & g^*(c_2)                                 
\end{tikzcd} \]
So $\alpha_{c_1}(\mathbb{M}(u)(a)) = (g^*u)^*\alpha_{c_2}(a)$, which is precisely the content of the naturality condition.\\

The cartesianness of $\alpha$ is actually inherited from that of $f^*$, as $ \alpha$ acts as a restriction of $f^*$ on each fiber in light of Proposition \ref{propcartesianstructure}.

\section{The $\mathbb{T}$-over-topos of a model in an arbitrary topos}

Now we turn to the construction of the over-topos $u_{M}:{\cal E}[M]\to {\cal E}$ associated with a $ \mathbb{T}$-model $M$ in an arbitrary topos $ \mathcal{E}$. Recall that the desired formula will be, for any $ \mathcal{E}$-topos $ g : \mathcal{G} \rightarrow \mathcal{E}$, 
\[ \textup{\bf Geom}_\mathcal{E} [g, u_{M}] \simeq \mathbb{T}[\mathcal{G}]/g^*(M). \]

We suppose that we are given a cartesian small subcanonical site of definition $(\mathcal{C},J)$ for $\cal E$. Note that the canonical embedding $i_{\cal C}$ of $\cal C$ into $\cal E$ identifies $\cal C$ with a separating set of objects for $\cal E$ and $J$ with the Grothendieck topology induced on it by the canonical topology on $\cal E$. \\

In the case where $ \mathcal{E}= \Set$, as $ \Set$ is generated by $1$ under coproducts, the global elements $  1 \rightarrow \llbracket \vec{x}^{\vec{A}}. \phi \rrbracket_M $ are sufficient to generate all the generalized elements $ X \simeq \coprod_X 1 \rightarrow \llbracket \vec{x}^{\vec{A}}. \phi \rrbracket_M$. In the general case, global elements must be replaced by generalized elements, possibly restricting to those whose domain is an object of $\cal C$, which we call \emph{basic generalized elements}. Indeed, we shall replace the category of global elements of $M$ by an indexed category of basic generalized elements of interpretations in $M$ of geometric formulas over the language of $\mathbb T$, that is, with the comma category $(i_{\cal C} \downarrow F_{M})$, where $ F_{M} : (\mathcal{C}_\mathbb{T}, J_\mathbb{T}) \rightarrow \mathcal{E}$ is the $ J_\mathbb{T}$-continuous cartesian functor sending a geometric formula-in-context $ \{ \vec{x}^{\vec{A}}.  \phi \} $ to its interpretation $ \llbracket \vec{x}^{\vec{A}}.  \phi \rrbracket_M$ in $M$. This defines a prestack 
\[  \begin{array}{rcl}
    \mathcal{C}^{\textup{op}} & \stackrel{\mathbb{M}}{\longrightarrow} & \textup{Cat} \\
    c  & \longmapsto & (c \downarrow F_{M}), \\
    c_1 \stackrel{u}{\rightarrow} c_2 & \longmapsto & (c_2 \downarrow F_{M}) \stackrel{u^*}{\rightarrow} (c_1 \downarrow F_{M})
\end{array} \]
where $u^* : (c_2 \downarrow F_{M}) {\rightarrow} (c_1 \downarrow F_{M})$ is the pre-composition functor sending some $ a : {c_2} \rightarrow \llbracket \vec{x}^{\vec{A}}.  \phi \rrbracket_M $ to $ a\circ u : {c_1} \rightarrow \llbracket \vec{x}^{\vec{A}}.  \phi \rrbracket_M  $.

\begin{proposition}
$\mathbb{M}$ is a cartesian stack on $({\cal C}, J)$. That is, for each $c $ in $ \mathcal{C}$, $(c \downarrow F_{M})$ is cartesian, and for any arrow $u : c' \rightarrow c$ in $\cal C$, the transition functor $(u \downarrow F_{M})$ is cartesian.
\end{proposition}

\begin{proof}
This is a consequence of $ F_{M}$ being cartesian: indeed any finite diagram \[
(a_i : c \rightarrow \llbracket \vec{x}_i^{\vec{A}_i} . \phi_i \rrbracket_M )_{i \in I}
\]   
in $(c\downarrow F_{M})$ defines a unique arrow $ (a_i )_{i \in I} : c \rightarrow \lim_{i \in I}  \llbracket \vec{x}_i^{\vec{A}_i} . \phi_i \rrbracket_M \simeq  F_{M}( \lim_{i \in I} \{ \vec{x}_i^{\vec{A}_i} . \phi_i\} )$. Now, for any arrow $ b :c \rightarrow  \llbracket \vec{x}^{\vec{A}} . \phi \rrbracket_M $ equipped with a cone $( [\theta_i] : b \rightarrow a_i)_{i \in I}$ in $(c \downarrow F_{M})$, there is a canonical arrow $ [\theta]_\mathbb{T} : \{ \vec{x}^{\vec{A}} . \phi \} \rightarrow \lim_{i \in I}  \{ \vec{x}_i^{\vec{A}_i} . \phi_i \}  $ in $ \mathcal{C}_\mathbb{T}$ providing a factorization of the $ [\theta_i]_\mathbb{T}$'s, whence $ \llbracket \theta \rrbracket_M $ also factorizes each $ \llbracket \theta_i \rrbracket_M$, thus providing a universal factorization of the cone in $(c \downarrow F_{M})$. For any $u : c' \rightarrow c$, it is easy to see that the composite $ (a_i)_{i \in I}\circ u : c' \rightarrow F_{M}(\lim_{i \in I} \{ \vec{x}_i^{\vec{A}_i} . \phi_i \})  $ is the limit of the $ a_i \circ u$'s in $(c' \downarrow F_{M})$ by the uniqueness of the factorization of a cone through the limit.
\end{proof}

As a consequence, the fibred category $ \int \mathbb{M}=(i_{\cal C} \downarrow F_{M})$ is cartesian. Its objects are the pairs $ (c,a)$, where $ c$ is an object of $\mathcal{C}$ and $a$ is a $c$-indexed basic element $a : c \rightarrow \llbracket \vec{x}^{\vec{A}}. \phi \rrbracket_M$, while an arrow $ (c_1, a_1) \rightarrow (c_2, a_2)$ is a pair $ (u,\lbrack \theta \rbrack_\mathbb{T})$ consisting of an arrow $u : c_1 \rightarrow c_2$ in $\cal C$ and an arrow $\lbrack \theta \rbrack_\mathbb{T} $ in $\mathcal{C}_\mathbb{T}$ such that $ a_2\circ u = \llbracket  \theta \rrbracket_M \circ a_1$:
\[ 
\begin{tikzcd}
c_1 \arrow[d, "a_1"'] \arrow[r, "u"]                                                               & c_2 \arrow[d, "a_2"]                                  \\
\llbracket \vec{x}_1^{\vec{a_1}}. \phi_1 \rrbracket_M \arrow[r, "\llbracket  \theta \rrbracket_M"] & \llbracket \vec{x}_2^{\vec{a_2}}. \phi_2 \rrbracket_M
\end{tikzcd} \]

In order to complete our generalization of the construction of section 2, we need to equip the category $ \int \mathbb{M} $ with a Grothendieck topology representing the analogue of the antecedents topology. This is not straightforward since, in the general context, there are no vertical covers a priori. Indeed, one would be tempted to define in each fiber $(c \downarrow F_{M})$ a topology of $c$-indexed antecedents, made of families of triangles
 
\[ 
\begin{tikzcd}
                                                                                                     & c \arrow[ld, "b\,"{above}] \arrow[rd, "\llbracket \theta_i \rrbracket_M \circ b = a"] &                                                 \\
\llbracket \vec{x}_i^{\vec{A}_i}. \phi_i \rrbracket_M \arrow[rr, "\llbracket \theta_i \rrbracket_M"] &                                                                           & \llbracket \vec{x}^{\vec{A}}. \phi \rrbracket_M
\end{tikzcd} \]
for $ ([\theta_i] : \{ \vec{x_i}^{\vec{A_i}}. \phi_i  \} \rightarrow  \{ \vec{x}^{\vec{A}}. \phi\} )_{i \in I}$ a family in ${\cal B}_{\mathbb T}(\{ \vec{x}^{\vec{A}}. \phi\})$ and $b$ ranging over all the antecedents of $a$ along the arrows $ \llbracket \theta_i \rrbracket_{i \in I}$. However this would not work, since, though $F_{M}$ sends the family $ ([\theta_i] : \{ \vec{x_i}^{\vec{A_i}}. \phi_i  \} \rightarrow  \{ \vec{x}^{\vec{A}}. \phi\} )_{i \in I}$ to an epimorphic family in $ \mathcal{E}$, there is no reason for a $c$-indexed basic element $a$ to have a $c$-indexed antecedent along any of the $ \llbracket \theta_i \rrbracket_M$'s. In order to take the horizontal components into account, we need to leave the domain of \ac antecedent' generalized elements vary among the objects of $\cal C$, and require them to cover the fibers $ \llbracket \theta_i \rrbracket_M ^{-1}(a)$ of $a$ along the arrows $ \llbracket \theta_i \rrbracket_M$:
 
\[ 
\begin{tikzcd}
\coprod_{i \in I}  \llbracket \theta_i \rrbracket_M ^{-1}(a) \arrow[d] \arrow[r, two heads] \arrow[rd, "\lrcorner", very near start, phantom]         & c \arrow[d, "a"]                                \\
\underset{i \in I}{\coprod} \llbracket \vec{x_i}^{\vec{A_i}}. \phi_i \rrbracket_M \arrow[r, "\underset{i \in I}{\coprod} \llbracket \theta_i \rrbracket_M", two heads] & \llbracket \vec{x}^{\vec{A}}. \phi \rrbracket_M
\end{tikzcd} \]
(Note that, as in the set-based setting, some $ \llbracket \theta_i \rrbracket_M ^{-1}(a)$ may be empty, though they jointly cover $c$). In light of the characterization of the antecedents topology in the set-based setting in terms of lifted topologies, provided by Corollary \ref{corantecedenttopology}, we are thus led to defining the antecedents topology on $\int \mathbb{M}$ as follows: 

\begin{definition}
The \emph{antecedents topology} $J^{\textup{ant}}_M$ on $\int \mathbb{M}$ is the Grothendieck topology on $(i_{\cal C}\downarrow F_{M})$ induced by the $(f_{M}, J_{\mathbb T})$-topology on $(1_{\cal E}\downarrow F_{M})$ (in the sense of Definition \ref{defliftedtopology}); that is, it has as a basis the collection of families
\[  \begin{tikzcd}
\big{(}(d_{ij}, b_{ij}) \arrow[]{rr}{(u_{ij}, [ \theta_{i}]_{\mathbb{T}})} && (c,a))_{i \in I, \; j \in J_i}
\end{tikzcd}\]
where $a:c\to \llbracket \vec{x}^{\vec{A}}. \phi \rrbracket_M$ is a basic generalized element, $([\theta_i ]_\mathbb{T} : \{ \vec{x}_i^{\vec{A}_i}. \phi_i \} \rightarrow  \{ \vec{x}^{\vec{A}}. \phi \})_{i \in I}$ is a family in ${\cal B}_{\mathbb T}(\{ \vec{x}^{\vec{A}}. \phi \})$, $(u_{ij} : d_{ij} \rightarrow c)_{j \in J_i}$ is a family of arrows in $\cal C$ (for each $i\in I$) and $(b_{ij} : d_{ij} \rightarrow  \llbracket  \vec{x}_i^{\vec{A}_i}. \phi_i\rrbracket_M)_{j \in J_i}$ is a family of arrows (for each $i\in I$) making the diagrams
\[  
\begin{tikzcd}
d_{ij} \arrow[r, "u_{ij}"] \arrow[d, "b_{ij}"']                                          & c \arrow[d, "a"]                     \\
\llbracket \vec{x}_i^{\vec{A}_i}. \phi_i \rrbracket_M \arrow[r, "\llbracket \theta_i \rrbracket_M"] & \llbracket \vec{x}^{\vec{A}}. \phi \rrbracket_M
\end{tikzcd}  \]
commutative and such that the family of arrows
\[ (\widetilde{b_{ij}} : d_{ij} \rightarrow \langle \llbracket \theta_i \rrbracket_M \rangle^{-1}(a))_{j \in J_i}\]
as in the following diagram is epimorphic :
\[ 
\begin{tikzcd}
d_{ij} \arrow[rrd, "u_{ij}", bend left=20] \arrow[rdd, "b_{ij}"', bend right=20] \arrow[rd, dashed, "\widetilde{b_{ij}}"] &                                                                                             &                                      \\
                                                             & \llbracket \theta_i \rrbracket_M^{-1}(a) \arrow[r, "\pi_{i}"] \arrow[d] \arrow[rd, "\lrcorner", phantom, near start] & c \arrow[d, "a"]                     \\
                                                                                              & \llbracket \vec{x}_i^{\vec{A}_i}. \phi_i \rrbracket_M \arrow[r, "\llbracket \theta_i \rrbracket_M"]    & \llbracket \vec{x}^{\vec{A}}. \phi \rrbracket_M
\end{tikzcd} \]

\begin{remarks}
\begin{enumerate}
    \item[(a)] The covering families in the definition of $J^{ant}_M$ are indexed by a dependent sum, with first a set indexing a basic cover of $J_\mathbb{T}$ and for each term of this cover, a basic covering family. To obtain a more conventional presentation, one can equivalently use a single indexing set $J$ and require a family of squares $(\{ \theta_i \}, u_i)_{i \in I}$ to induce an epimorphism $ \langle ( b_i, u_i) \rangle_{i \in I}$ and have $ ( \theta_i)_{i \in I}$ in $J_\mathbb{T}$. In particular, for a presentation as above, take $J = \coprod_{i \in I} J_i$ and impose $ \{\theta_{(i,j)} \} := \{\theta_i \}$. 
    
    \item[(b)] As in the set-based case, the $\llbracket \theta_i \rrbracket_{M}$'s are only jointly epimorphic, so a generalized element $a:c\to \llbracket \vec{x}^{\vec{A}}. \phi \rrbracket_M$ may have no antecedent along a chosen $\theta_i$, that is, the corresponding set $J_i$ may be empty. However, antecedent families live over \emph{subfamilies} of $J_\mathbb{T}$-covering families, as shown by Proposition \ref{propcomorphismtosyntacticsite} below.
\end{enumerate}
\end{remarks}
\end{definition}
 
\begin{proposition}\label{propcomorphismtosyntacticsite}
There is a comorphism of sites $ (\int \mathbb{M}, J^{\textup{ant}}_{M}) \rightarrow (\mathcal{C}_\mathbb{T}, J_\mathbb{T})$ sending a basic generalized element $ a :c \rightarrow \llbracket \vec{x}^{\vec{A}}. \phi \rrbracket_M $ to the underlying sort $ \{ \vec{x}^{\vec{A}}. \phi \}$.
\end{proposition}

\begin{proof}
This follows immediately from the characterization of $J^{\textup{ant}}_{M}$ as the Grothendieck topology on $(i_{\cal C}\downarrow F_{M})$ induced by the $(f_{M}, J_{\mathbb T})$-topology on $(1_{\cal E}\downarrow F_{M})$.
\end{proof}

\begin{remark}
The comorphism of sites of Proposition \ref{propcomorphismtosyntacticsite} is \emph{not} a fibration, unlike its topos-theoretic extension provided by the canonical projection functor $(1_{\cal E}\downarrow {f_{M}}^{\ast}) \to \Sh({\cal C}_{\mathbb T}, J_{\mathbb T})$. This explains why the antecedent topology is \emph{not} a lifted topology but rather a topology induced on a smaller subcategory by a lifted topology existing at the topos level, and hence that its description is more involved than that for its topos-theoretic extension. This is an illustration of the importance of developing \emph{invariant constructions} at the topos-theoretic level and of investigating  only later how such notions can be described at the level of sites.     
\end{remark}

Since $J^{\textup{ant}}_{M}$ is the Grothendieck topology on $(i_{\cal C}\downarrow F_{M})$ induced by the $(f_{M}, J_{\mathbb T})$-lifted topology on $(1_{\cal E}\downarrow F_{M})$, the canonical projection functor $\pi_{M}:{\int \mathbb{M}} \to {\cal C}$ to $\cal C$ is a comorphism of sites $(\int \mathbb{M}, J^{\textup{ant}}_{M}) \rightarrow (\mathcal{C},J)$.

We are now ready to define the over-topos at $M$ in the general setting:

\begin{definition}\label{overtopos}
The $ \mathbb{T}$-over-topos of $\mathcal{E}$ at $M$ is defined as \[ \mathcal{E}[M] \simeq \Sh(\textstyle\int \mathbb{M}, J^{\textup{ant}}_{M}) \stackrel{u_{M}}{\longrightarrow} \mathcal{E} \]
where $u_{M}$ is the geometric morphism induced by the comorphism of sites $\pi_{M} : (\int \mathbb{M}, J^{\textup{ant}}_\mathbb{M}) \rightarrow (\mathcal{C},J)$.
\end{definition}

\begin{remarks}

\begin{enumerate}
\item[(a)] The inverse image of $u_{M}$ is the pre-composition sending $ c$ to $ \Hom_{\mathcal{C}}(\pi_{M}(-),c)$ (since $\pi_{M}$ is continuous by Theorem 4.44 \cite{OCdenseness}), while the direct image is the restriction to sheaves of the right Kan extension along $\pi_{M}$.

\item[(b)] In the set-based case, choosing as a presentation site the category $ \{ * \}$ with the trivial topology on it, there is only one fiber, and the antecedents topology on $ \int M$ of Definition \ref{antecedents topology} is completely concentrated in it.
\end{enumerate}
\end{remarks}

\begin{theorem}\label{thmovertopos}
The $\cal E$-topos $u_{M}: \mathcal{E}[M]\to {\cal E}$  satisfies the universal property of the $\mathbb{T}$-over-topos at $ M$: that is, for any $\cal E$-topos $ g : \mathcal{G} \rightarrow \mathcal{E}$ there is an equivalence of categories
\[ \textup{\bf Geom}_{\mathcal{E}}[g, \mathcal{E}[\mathbb{M}]] \simeq \mathbb{T}[\mathcal{G}]/g^*M \]
natural in $g$.
\end{theorem}

\begin{proof}
The proof naturally generalizes that for a set-based model. 

In one direction, suppose that $f$ is a geometric morphism over $ \mathcal{E}$ from $g$ to $u_{M}$:
\[ 
\begin{tikzcd}
\mathcal{G} \arrow[rr, "f"] \arrow[rd, "g"'] &             & {\mathcal{E}[\mathbb{M}]} \arrow[ld, "u_{M}"] \\
                                             & \mathcal{E} &                                                     
\end{tikzcd} \]
This defines a $J^{\textup{ant}}_{M}$-continuous, cartesian functor 
\[ \textstyle\int \mathbb{M} \stackrel{f^*}{\longrightarrow} \mathcal{G}. \]
By Theorem \ref{thmGiraud}, it also induces a morphism of cartesian stacks over $\mathcal{E}$ 
\[ \mathbb{M} \stackrel{f}{\longrightarrow} \mathcal{G}/g^* \]
whose components 
\[ (c \downarrow F_{M}) \stackrel{f_c}{\longrightarrow} \mathcal{G}/g^*c \]
(for each $c $ in $ \mathcal{C}$) are cartesian functors sending a given basic generalized element $ a : c \rightarrow \llbracket \vec{x}^{\vec{A}}. \phi \rrbracket_M $ to a certain arrow $ f_c(a) : N_{(c,a)} \rightarrow g^*c $ in ${\cal G}$. We want to associate with $f$ a morphism $\tilde{f} : N \rightarrow g^*M$ in $\mathbb{T}[\mathcal{G}]$. For each $  \{ \vec{x}^{\vec{A}}. \phi\} $ in $ \mathcal{C}_\mathbb{T}$, the category $ (i_{\cal C}\downarrow \llbracket \vec{x}^{\vec{A}}. \phi \rrbracket_M)$ of basic generalized elements is small; we can thus define 
\[ N_{ \{ \vec{x}^{\vec{A}}. \phi\}} := \underset{(i_{\cal C}\downarrow \{ \vec{x}^{\vec{A}}. \phi\})}{\colim} N_{(c,a)} \simeq \underset{c \in \mathcal{C}}{\colim} \coprod_{a\in \Hom_{\mathcal{E}}(c, \llbracket \vec{x}^{\vec{A}}. \phi \rrbracket_M)} N_{(c,a)} \]
and, by using the universal property of the colimit, an arrow $ f_{\{ \vec{x}^{\vec{A}}. \phi\}}$ as in the following diagram (where the $j_a$'s are the legs of the colimit cocone): 
\[ 
\begin{tikzcd}[column sep=large]
N_{(c,a)} \arrow[r, "f_c(a)"]  \arrow[d, "j_a"'] & g^*c \arrow[d, "g^*(a)"]                              \\
N_{\{ \vec{x}^{\vec{A}}. \phi\}} \arrow[r, dashed, "  f_{\{ \vec{x}^{\vec{A}}. \phi\}} "]   & \llbracket \vec{x}^{\vec{A}}. \phi \rrbracket_{g^*M }
\end{tikzcd} \]

Let us first check this yields a $ J_\mathbb{T}$-continuous cartesian functor $ N : (\mathcal{C}_\mathbb{T}, J_\mathbb{T}) \rightarrow \mathcal{G}$, showing how the above assignment on objects naturally extends to arrows. 

Given an arrow $[\theta]_{\mathbb T}:\{\vec{x_1}^{\vec{A_{1}}}. \phi_1\} \to \{\vec{x_2}^{\vec{A_{2}}}. \phi_2\}$ in ${\cal C}_{\mathbb T}$, composing with $\llbracket \theta \rrbracket_M$ allows one to associate with any generalized element $ a : c \rightarrow \llbracket \vec{x_1}^{\vec{A_1}}. \phi \rrbracket_M$ a generalized element $\llbracket \theta \rrbracket_M\circ a:c \to \llbracket \vec{x_2}^{\vec{A_2}}. \phi \rrbracket_M$. So by the functoriality of $f_{c}$ we have:
\[ 
\begin{tikzcd}
{N_{(c,a)}} \arrow[rr, "f_c(\lbrack \theta \rbrack_\mathbb{T})"] \arrow[rd, "f_c(a)"'] &        & {N_{(c, \llbracket \theta \rrbracket_M \circ a)}} \arrow[ld, "f_c(\llbracket \theta \rrbracket_M \circ a)"] \\
                                                                                         & g^*(c) &
\end{tikzcd} \]

We can thus define $N_{[\theta]_{\mathbb T}}:N_{\{ \vec{x_1}^{\vec{A_1}}. \phi_1\}}  \to N_{\{ \vec{x_2}^{\vec{A_2}}. \phi_2\}}$ as the arrow  determined by the universal property of the colimit by the requirement that all the diagrams of the form
\[
\begin{tikzcd}[column sep=large]
N_{(c,a)} \arrow[r, "f_c(\lbrack \theta \rbrack_\mathbb{T})"]  \arrow[d, "j_a"'] & N_{(c, \llbracket \theta \rrbracket_M \circ a)} \arrow[d, "j_{\llbracket \theta \rrbracket_M \circ a}"]                              \\
N_{\{ \vec{x_1}^{\vec{A_1}}. \phi_1\}} \arrow[r, dashed, "N_{[\theta]_{\mathbb T}}"]   & N_{\{ \vec{x_2}^{\vec{A_2}}. \phi_2\}} 
\end{tikzcd} \]
should commute.

The fact that $N$ is cartesian follows at once from the fact that the functor $f^{\ast}:{\int {\mathbb M}}\to {\cal G}$ is, in light of the definition of $N$ in terms of colimits and of the stability of these under pullback. So it remains to prove its $J_{\mathbb T}$-continuity.

Let $([\theta_i ]_\mathbb{T} : \{ \vec{x}_i^{\vec{A}_i}. \phi_i \} \rightarrow  \{ \vec{x}^{\vec{A}}. \phi \})_{i \in I}$ be a family in ${\cal B}_{\mathbb T}(\{ \vec{x}^{\vec{A}}. \phi \})$; we want to show that the family $\big{(}N_{[\theta_i ]_\mathbb{T}} : N_{\{ \vec{x}_i^{\vec{A}_i}. \phi_i \}} \rightarrow  N_{\{ \vec{x}^{\vec{A}}. \phi \}}\big{)}_{i \in I}$ is epimorphic. First, we notice that this condition can be conveniently phrased in terms of basic generalized elements, as follows: for any basic generalized element $w:e\to N_{\{ \vec{x}^{\vec{A}}. \phi \}}$ there are an epimorphic family $\{u_{k}:e_k \to e \mid k\in K\}$ lying in $\cal C$ and for each $k\in K$ an element $i_{k}\in I$ and a basic generalized element $w_{k}:e_{k}\to N_{\{ \vec{x_{i_{k}}}^{\vec{A_{i_{k}}}}. \phi_{i_{k}} \}}$ such that the diagram 
\[  
\begin{tikzcd}
e_{k}\arrow[r, "u_{k}"] \arrow[d, "w_{k}"] & e \arrow[d, "w"]                     \\
N_{\{ \vec{x_{i_{k}}}^{\vec{A_{i_{k}}}}. \phi_{i_{k}} \}} \arrow[r, "N_{[\theta_{i_{k}}]_{\mathbb T}}"] &  N_{\{ \vec{x}^{\vec{A}}. \phi \}}
\end{tikzcd}  \]
commutes. Since colimits yield epimorphic families in a topos, in light of the definition of $N$ we can further rewrite this condition as follows: for any basic generalized element $w:e\to N_{(c, a)}$, where $a:c\to \llbracket \vec{x}^{\vec{A}}. \phi \rrbracket_M$, there are an epimorphic family $\{u_{k}:e_k \to e \mid k\in K\}$ lying in $\cal C$ and for each $k\in K$ an element $i_{k}\in I$, an object $(d_k, b_{k})$ of the category $(i_{\cal C}\downarrow \llbracket \vec{x_{i_{k}}}^{\vec{A_{i_{k}}}}. \phi_{i_{k}} \rrbracket_M)$ and an arrow $w_{k}:e_{k}\to N_{(d_k, b_k)}$ such that the diagram 
\[  
\begin{tikzcd}
e_{k}\arrow[r, "u_{k}"] \arrow[d, "w_{k}"] & e \arrow[d, "w"]                     \\
N_{(d_k, b_{k})} \arrow[r, "N_{[\theta_{i_{k}}]_{\mathbb T}}"] &  N_{(c, a)}
\end{tikzcd}  \] 
commutes.

Now, given a generalized element $ a : c \rightarrow \llbracket \vec{x}^{\vec{A}}. \phi \rrbracket_M$, we may obtain a covering family for the antecedent topology $J^{\textup{ant}}_{M}$ (with respect to our original $J_{\mathbb T}$-cover) by covering each of the fibers $\llbracket \theta_i \rrbracket_M^{-1}(a)$ by basic generalized elements $\widetilde{b_{ij}}:d_{ij}\to \llbracket \theta_i \rrbracket_M^{-1}(a)$:
\[ 
\begin{tikzcd}
d_{ij} \arrow[rrd, "u_{ij}", bend left=20] \arrow[rdd, "b_{ij}"', bend right=20] \arrow[rd, dashed, "\widetilde{b_{ij}}"] &                                                                                             &                                      \\
                                                                                              & \llbracket \theta_i \rrbracket_M^{-1}(a) \arrow[r, "\pi_{i}"] \arrow[d, "\llbracket \theta_i \rrbracket^{\ast}_M(a)"] \arrow[rd, "\lrcorner", phantom, very near start] & c \arrow[d, "a"]                     \\
                                                                                              & \llbracket \vec{x}_i^{\vec{A}_i}. \phi_i \rrbracket_M \arrow[r, "\llbracket \theta_i \rrbracket_M"]    & \llbracket \vec{x}^{\vec{A}}. \phi \rrbracket_M
\end{tikzcd} \]

The functor $f^{\ast}$ being $J^{\textup{ant}}_{M}$-continuous, $f^{\ast}$ sends this family to an epimorphic family
\[
\big{(} f(u_{ij}, [\theta_i]_{\mathbb T}): N_{(d_{ij}, b_{ij})} \to N_{(c, a)} \big{)}_{ i\in I,\, j\in J_i}
\]
in $\cal G$; but this clearly implies our thesis.

Conversely, let $ f : N \rightarrow g^*M$ be a morphism of $ \mathbb{T}$-models in $ \mathcal{G}$. We can regard $f$ as a $ \mathcal{C}_\mathbb{T}$-indexed family $\{f_{\{\vec{x}^{\vec{A}}. \phi \} }: \llbracket \vec{x}^{\vec{A}}. \phi \rrbracket_N  \to \llbracket \vec{x}^{\vec{A}}. \phi \rrbracket_{g^*M }\}$ of morphisms in $ \mathcal{G}$ (subject to the naturality conditions). For each basic generalized element $ a : c \rightarrow \llbracket \vec{x}^{\vec{A}}. \phi \rrbracket_{M }$, we define $ N_{(c,a)}$ as the following pullback: 
\[ 
\begin{tikzcd}[column sep=large]
N_{(c,a)} \arrow[r] \arrow[d] \arrow[rd, "\lrcorner", very near start, phantom]                                  & g^*c \arrow[d, "g^*(a)"]                              \\
\llbracket \vec{x}^{\vec{A}}. \phi \rrbracket_N \arrow[r, "f_{\{\vec{x}^{\vec{A}}. \phi \} }"] & \llbracket \vec{x}^{\vec{A}}. \phi \rrbracket_{g^*M }
\end{tikzcd} \]
Given a morphism $(u, \lbrack \theta \rbrack_\mathbb{T} ) : (c_1, a_1) \rightarrow (c_2, a_2)$ in $ \int \mathbb{M}$, we have by the naturality of $f$ (seen as a morphism of $ J_\mathbb{T}$-continuous cartesian functors) the following commutative square: 
\[ 
\begin{tikzcd}[column sep= large]
\llbracket \vec{x}_1^{\vec{a_1}}. \phi_1 \rrbracket_N  \arrow[d, "\llbracket \theta \rrbracket_{N}"'] \arrow[r, "f_{ \{ \vec{x}_1^{\vec{a_1}}. \phi_1  \} } "] & \llbracket \vec{x}_1^{\vec{a_1}}. \phi_1 \rrbracket_{g^*M} \arrow[d, "\llbracket \theta \rrbracket_{g^*M}"] \\
\llbracket \vec{x}_2^{\vec{a_2}}. \phi_2 \rrbracket_N  \arrow[r, "f_{ \{ \vec{x}_2^{\vec{a_2}}. \phi_2  \} } "]                                                & \llbracket \vec{x}_2^{\vec{a_2}}. \phi_2 \rrbracket_{g^*M}                                                 
\end{tikzcd}\]

We can thus define ${N_{(u,\lbrack \theta \rbrack_\mathbb{M})}}:N_{(c_1, a_1)} \to N_{(c_2, a_2)}$ as the unique arrow making the following diagram (where the front and back faces are pullback) commute:
 
\[ 
\begin{tikzcd}[column sep=small, row sep=small]
{N_{(c_1, a_1)}} \arrow[dd] \arrow[rr] \arrow[rd, "{N_{(u,\lbrack \theta \rbrack_\mathbb{M})}}", dashed]                  &                                                                                                                  & g^*(c_1) \arrow[rd, "g^*(u)"] \arrow[d, "a_1", no head]                                                      &                                                             \\
                                                                                                                          & {N_{(c_2, a_2)}} \arrow[dd] \arrow[rr] \arrow[rrdd, "\lrcorner", very near start, phantom]                                        & {} \arrow[d]                                                                                                 & g^*(c_2) \arrow[dd, "a_2"]                                  \\
\llbracket \vec{x}_1^{\vec{a_1}}. \phi_1 \rrbracket_N  \arrow[rd, "\llbracket \theta \rrbracket_{N}"'] \arrow[r, no head] & {} \arrow[r, "f_{ \{ \vec{x}_1^{\vec{a_1}}. \phi_1  \} } "]                                                      & \llbracket \vec{x}_1^{\vec{a_1}}. \phi_1 \rrbracket_{g^*M} \arrow[rd, "\llbracket \theta \rrbracket_{g^*M}"] &                                                             \\
                                                                                                                          & \llbracket \vec{x}_2^{\vec{a_2}}. \phi_2 \rrbracket_N  \arrow[rr, "f_{ \{ \vec{x}_2^{\vec{a_2}}. \phi_2  \} } "] &                                                                                                              & \llbracket \vec{x}_2^{\vec{a_2}}. \phi_2 \rrbracket_{g^*M} 
\end{tikzcd} \]

This yields a functor  
\[ \textstyle\int \mathbb{M} \stackrel{N_{(-)}}{\longrightarrow} \mathcal{G}. \]

We want to show that this functor is cartesian and $ J^{\textup{ant}}_{M}$-continuous. The fact that it is cartesian follows easily from the fact that the functor $F_{N}:{\cal C}_{\mathbb T}\to {\cal G}$ corresponding to the model $N$ is, in light of the construction of finite limits in the category ${\int \mathbb{M}}$ provided by Proposition \ref{propcartesianstructure}.

Concerning $J^{\textup{ant}}_{M}$-continuity, we shall consider the extension of $N_{(-)}$ to the category $(1_{\cal E}\downarrow F_{M})$ and show its continuity with respect to the extended topology, as the latter admits a more natural characterization as the $(f_{M}, J_{\mathbb T})$-lifted topology; recall that this topology has a basis consisting of multicomposites of covering families of horizontal arrows for the fibration $r_{f_{M}}$ and of covering families of horizontal arrows for the fibration $t_{f_{M}}$, so continuity can be checked separately with respect to each of these families.

Let us start by showing the continuity with respect to the covering families of horizontal arrows for $t_{f_{M}}$; this can be checked without any problems in terms of the site of definition $({\cal C}, J)$ for $\cal E$. For a $J$-covering family $ (u_i : c_i \rightarrow c)_{i \in I}$ and a generalized element $ a : c \rightarrow \llbracket \vec{x}^{\vec{A}}. \phi \rrbracket_M$, consider the following diagram: 
\[ 
\begin{tikzcd}
N_{(c_i, a\circ u_i)}  \arrow[d, "{N_{(u_i, 1_{\{\vec{x}^{\vec{A}}. \phi  \} })}}"'] \arrow[r]              & g^*c_i \arrow[d, "g^*u_i"]                            \\
{N_{(c,a)}} \arrow[r] \arrow[d] \arrow[rd, "\lrcorner", very near start, phantom] & g^*c \arrow[d, "g^*(a)"]                              \\
\llbracket \vec{x}^{\vec{A}}. \phi \rrbracket_{N} \arrow[r, "f_{\{\vec{x}^{\vec{A}}. \phi \}}"', no head]                      & \llbracket \vec{x}^{\vec{A}}. \phi \rrbracket_{g^*M }
\end{tikzcd} \]
The lower square and the outer rectangle are pullbacks, whence by the pullback lemma the upper square is also a pullback. Since $({\cal C}, J)$ is a site of definition for $\cal E$, $g^{\ast}$ sends $J$-covering families to epimorphic families; so we have a pullback of epimorphisms 
\[ 
\begin{tikzcd}
{\underset{i \in I}{\coprod} N_{(c_i, a\circ u_{i})}} \arrow[d, "{\langle N_{(u_i, 1_{ \{ \vec{x}^{\vec{A}}. \phi  \} })} \rangle_{i \in I}}"', two heads] \arrow[r] \arrow[rd, "\lrcorner", very near start, phantom] & \underset{i \in I}{\coprod} g^*c_i \arrow[d, "\langle g^*u_i\rangle_{i\in I}", two heads] \\
{N_{(c,a)}} \arrow[r]                                                                                                                                                                               & g^*c                                                             
\end{tikzcd} \]
ensuring that $N_{(-)}$ sends the horizontal covering family 
\[
\big{(} (u_{i}, 1_{\{\vec{x}^{\vec{A}}. \phi\}}):(c_i, \, (\{\vec{x}^{\vec{A}}. \phi\}, a\circ u_i)) \to (c, \, (\{\vec{x}^{\vec{A}}. \phi\}, a))\big{)}_{i\in I} 
\]
to an epimorphic family. 

Now we turn to the proof of the continuity of $N_{(-)}$ with respect to the covering families of horizontal arrows for the fibration $r_{f_{M}}$. 

Given a family $\big{(}[\theta_i]:\{\vec{x_i}. \phi_i\} \to \{\vec{x}. \phi\}\big{)}_{i\in I}$ in ${\cal B}_{\mathbb T}$ and a generalized element $a:c\to \llbracket \vec{x}^{\vec{A}}. \phi \rrbracket_{M}$, we want to prove that $N_{(-)}$ sends the family
\[
\big{(} (\pi_{i}, [\theta_i]):(\llbracket \theta_i \rrbracket_M^{-1}(a), (\{\vec{x_i}. \phi_i\}, \llbracket \theta_i \rrbracket_{M}^{\ast}(a))) \to (c, ( \{\vec{x}. \phi\}, a))\big{)}_{i\in I}  
\]
in $\int {\mathbb M}$ to an epimorphic family.

Consider, for each $i\in I$, the following diagram:

\[ 
\begin{tikzcd}[column sep=small, row sep=small]
{N_{(\llbracket \theta_i \rrbracket^{-1}(a), \llbracket \theta_i \rrbracket^{\ast}(a))}} \arrow[dd] \arrow[rr] \arrow[rd, "{N_{(\pi_{i},\lbrack \theta_i \rbrack)}}", dashed]                  &                                                                                                                  & g^*(\llbracket \theta_i \rrbracket^{-1}(a)) \arrow[rd, "g^*(\pi_i)"] \arrow[d, "\llbracket \theta_i \rrbracket^{\ast}(a))", no head]                                                      &                                                             \\
                                                                                                                          & {N_{(c, a)}} \arrow[dd] \arrow[rr] \arrow[rrdd, "\lrcorner", very near start, phantom]                                        & {} \arrow[d]                                                                                                 & g^*(c) \arrow[dd, "a"]                                  \\
\llbracket \vec{x_i}^{\vec{A_i}}. \phi_i \rrbracket_N  \arrow[rd, "\llbracket \theta_i \rrbracket_{N}"'] \arrow[r, no head] & {} \arrow[r, "f_{ \{ \vec{x_i}^{\vec{A}_i}. \phi_i  \} } "]                                                      & \llbracket \vec{x_i}^{\vec{A_i}}. \phi_i \rrbracket_{g^*M} \arrow[rd, "\llbracket \theta_{i} \rrbracket_{g^*M}"] &                                                             \\
                                                                                                                          & \llbracket \vec{x}^{\vec{A}}. \phi \rrbracket_N  \arrow[rr, "f_{ \{ \vec{x}^{\vec{A}}. \phi  \} } "] &                                                                                                              & \llbracket \vec{x}^{\vec{A}}. \phi \rrbracket_{g^*M} 
\end{tikzcd} \]

The front square is a pullback by definition of $N_{(c, a)}$, the back square is a pullback by definition of $N_{(\llbracket \theta_i \rrbracket^{-1}(a), \llbracket \theta_i \rrbracket^{\ast}(a))}$ and the right-hand lateral one is a pullback since $g^{\ast}$ is cartesian. So the composite of the back square with the right-hand lateral square is a pullback, and hence by the pullback lemma the left-hand lateral square is also a pullback (as the front square is a pullback). Our thesis thus follows from the fact that, since $N$ is a $\mathbb T$-model by our hypotheses, the family
\[
\big{(}\llbracket \theta_i \rrbracket_{N}: \llbracket \vec{x_i}^{\vec{A_{i}}}. \phi_i \rrbracket_{N} \to \llbracket \vec{x}^{\vec{A}}. \phi \rrbracket_{N}\big{)}_{i\in I}
\]
is epimorphic.

Finally, proving that the functors defined above are mutually quasi-inverses is a straightforward exercise which we leave to the reader. 
\end{proof}

To conclude this section, we remark that our construction of the over-topos is functorial; that is, a morphism of $ \mathbb{T}$-models naturally induces a canonical morphism between the corresponding $\mathbb{T}$-over-toposes. Let $ f : M_1 \rightarrow M_2$ be a morphism in $ \mathbb{T} [\mathcal{E}]$. This is a same as a natural transformation 
\[ 
\begin{tikzcd}
{(\mathcal{C}_\mathbb{T}, J_\mathbb{T}) } \arrow[rr, "M_1", bend left=20, start anchor=10, ""{name=U, below, inner sep=0.1pt}] \arrow[rr, "M_2"', bend right=20, ""{name=D, inner sep=0.1pt}, start anchor=350] &  & \mathcal{E} \arrow[from=U, to=D, Rightarrow]{}{f}
\end{tikzcd} \]
whose components
\[ \begin{tikzcd}
\llbracket \vec{x}^{\vec{A}}. \phi \rrbracket_{ M_1 } \arrow[]{rr}{ f_{ \{ \vec{x}^{\vec{A}}. \phi \} }} &&  \llbracket \vec{x}^{\vec{A}}. \phi \rrbracket_{ M_2 }
\end{tikzcd}    \]
are indexed by the objects of $\mathcal{C}_\mathbb{T}$. This induces an indexed functor $ \alpha : \mathbb{M}_1 \Rightarrow \mathbb{M}_2$ assigning to each object $c $ of $\mathcal{C}$ the functor 
\[ \begin{tikzcd}
 (c \downarrow F_{M_1}) \arrow[]{rr}{\alpha_c = (c \downarrow f)} && (c \downarrow F_{M_2})
\end{tikzcd}  \]
sending a basic global element $ a : c \rightarrow \llbracket \vec{x}^{\vec{A}}. \phi \rrbracket_{ M_1 } $ to $f_{ \{ \vec{x}^{\vec{A}}. \phi \}} \circ a : c \rightarrow  \llbracket \vec{x}^{\vec{A}}. \phi \rrbracket_{ M_2 }$. This clearly yields a comorphism of sites 
\[
((i_{\cal C} \downarrow F_{M_1}), J^{\textup{ant}}_{M_1}) \to ((i_{\cal C}\downarrow F_{M_2}), J^{\textup{ant}}_{M_{2}})
\]
over $\cal C$, which thus induces a geometric morphism 
\[
{\cal E}[\alpha]:{\cal E}[{\mathbb M}_{1}] \to {\cal E}[{\mathbb M}_{2}]
\]
over $\cal E$ between the associated over-toposes.

\section{The $\mathbb{T}$-over-topos as a bilimit }

In this section we explore the $2$-dimensional aspects of the construction of the $ \mathbb{T}$-over-topos, exhibiting it as a bilimit in the bicategory of Grothendieck toposes; we shall also discuss its relation with the notion of \emph{totally connected topos}. We will see that, when $M$ is a model of $\mathbb T$ in $\Set$, hence a point $f_{M}:\Set\to \Set[\mathbb{T}]$ of $\Set[\mathbb{T}]$ over $\Set$, this yields a totally connected topos over $\Set$, which coincides with the colocalization of $\Set[\mathbb{T}]$ at $M$ studied in \cite{elephant}[Theorem C3.6.19]; indeed, in this case the over-topos $u_{M}:\Set[M] \to \Set$ admits a $\Set$-point $s_{M}:\Set\to \Set[M]$ providing a factorization of $f_{M}:\Set\to \Set[\mathbb{T}]$ as $s_{M}:\Set\to \Set[M]$ followed by the geometric morphism $f_{M}\circ u_{M}:\Set[M]\to \Set[{\mathbb T}]$, which satisfies the universal property of the colocalization of $\Set[\mathbb{T}]$ at $M$. Note, however, that, for a section $s:{\cal S}\to {\cal E}$ of a $\cal S$-topos $p:{\cal E}\to {\cal S}$, the colocalization of $\cal E$ at $s$, as defined in \cite{elephant} differs in general from the over-topos of ${\cal S}$ at $s$, since the universal property of the former provides an equivalence between $\textup{\bf Geom}_{\cal S}({\cal F}, u_{s})$ and $\textup{\bf Geom}_{\cal S}({\cal F}, {\cal E})\slash (s\circ g)$ (where $\cal E$ is regarded as a $\cal S$-topos via $p$), while the universal property of the latter provides and equivalence between $\textup{\bf Geom}_{\cal S}({\cal F}, u_{s})$ and $\textup{\bf Geom}({\cal F}, {\cal E})\slash (s\circ g)$. 

Indeed, we can geometrically define the \emph{over-topos of $\cal F$ at a geometric morphism $f:{\cal F}\to {\cal E}$} as a $\cal F$-topos $u_{f}:{\cal F}[f]\to {\cal F}$ with a morphism $\xi_{f}:{\cal F}[f]\to {\cal E}$ satisfying the following universal property: for any $\cal F$-topos $g:{\cal G}\to {\cal F}$, there is an equivalence  
	\[
	\textup{\bf Geom}_{\cal F}({\cal F}, u_{f})\simeq \textup{\bf Geom}({\cal F}, {\cal E})\slash (f\circ g)
	\]
	natural in $g$, induced by composition with $\xi_{f}$. 

Note that the condition $p\circ s=1$ is not sufficient in general to ensure that $\xi_{s}$ is a morphism over $\cal S$ (i.e., that $p\alpha_{s}:  p\circ \xi_{s}\to p\circ s\circ u_{s}\cong u_{s}$, where $\alpha_{s}:\xi_{s}\to s\circ u_{s}$ is the canonical geometric transformation, is an isomorphism), which explains why the colocalization of $\cal E$ at $s$ is not in general equivalent to the over-topos of ${\cal S}$ at $s$. Of course, the above definition can be readily generalized to yield a notion of over-topos relative to an arbitrary base topos; if we do so and regard $s$ as a morphism of $\cal S$-toposes then the $\cal S$-over-topos of ${\cal S}$ at $s$ coincides precisely with the colocalization of $\cal E$ at $s$.

Next, we recall some classical notions about bilimits in the bicategory of Grothendieck toposes; the reader may refer to \cite{elephant}[B4.1] and \cite{Zoghaib} for more details.

\begin{definition}
	Let $ \mathcal{K}$ be a bicategory and $ f: X \rightarrow Z$, $g : Y \rightarrow Z$ two 1-cells with common codomain. \begin{itemize}
		\item  A \emph{bipullback} of $ f, \,g$ is an invertible 2-cell
		\[
		\begin{tikzcd}
		X \times^{f,g}_Z Y \arrow[d, "f^*g"'] \arrow[r, "g^*f"] \arrow[rd, "\alpha \atop \simeq", phantom] & Y \arrow[d, "g"] \\
		X \arrow[r, "f"']                                                                            & Z               
		\end{tikzcd} \]
		such that for any $0$-cell $W$ and any invertible 2-cell 
		\[
		\begin{tikzcd}
		W \arrow[r, "t"] \arrow[d, "s"'] \arrow[rd, "\beta \atop \simeq", phantom] & Y \arrow[d, "g"] \\
		X \arrow[r, "f"']                                                          & Z               
		\end{tikzcd}\]
		there is an arrow $ u_\beta :  W \rightarrow X \times_Z Y$, unique up to unique invertible 2-cell, and a pair of invertible 2-cells $ \alpha_s, \, \alpha_t $ such that:  
		\[  
		\begin{tikzcd}
		W \arrow[r, "t"] \arrow[d, "s"'] \arrow[rd, "\beta \atop \simeq", phantom] & Y \arrow[d, "g"] \\
		X \arrow[r, "f"']                                                          & Z               
		\end{tikzcd} =
		\begin{tikzcd}
		W \arrow[rrd, "t", bend left=20, ""{name=Ut, inner sep=0.1pt, below}] \arrow[rdd, "s"', bend right=20, ""{name=Us, inner sep=0.1pt}] \arrow[rd, "{u_\beta}" description, dashed, ""{name= Dt}, ""{name=Ds, below}] &                                                                                              &                  \\ \arrow[from = Ut, to=Dt, phantom]{}{\alpha_t \atop \simeq}
		\arrow[from = Us, to=Ds, phantom]{}{\alpha_s \atop \simeq}
		& X \times^{f,g}_Z Y \arrow[d, "f^*g"'] \arrow[r, "g^*f"] \arrow[rd, "\alpha \atop \simeq", phantom] & Y \arrow[d, "g"] \\
		& X \arrow[r, "f"']                                                                            & Z               
		\end{tikzcd} \]
		In other words, the bipullback induces for each $ W$ an equivalence in $\textup{Cat}$
		\[ \mathcal{K}[ W, X \times^{f,g}_Z Y] \simeq \mathcal{K}[W,X] \times^{ \mathcal{K}[W,f], \mathcal{K}[W,g]}_{\mathcal{K}[W,Z]} \mathcal{K}[W,Y]. \]
		In the sequel, the universal 2-cell in the bipullback will be considered implicit, and we shall use the same notation as for a pullback in a 1-category.
		\item A \emph{bicomma object} is a 2-cell
		\[ 
		\begin{tikzcd}
		(f \downarrow g) \arrow[r, "p_g"] \arrow[d, "p_f"'] & Y \arrow[d, "g", ""{name=U, inner sep=0.1pt, below}] \\
		X \arrow[r, "f"', ""{name=D, inner sep=0.1pt}]  \arrow[from= D, to=U, Rightarrow, bend left=20]{}{\lambda_{f,g}}                               & Z               
		\end{tikzcd} \]
		such that for any other 2-cell
		\[ \begin{tikzcd}
		W \arrow[r, "t"] \arrow[d, "s"']  & Y \arrow[d, "g", ""{name=U, inner sep=0.1pt, below}] \\
		X \arrow[r, "f"', ""{name=D, inner sep=0.1pt}]          \arrow[from= D, to=U, Rightarrow, bend left=20]{}{\mu}                                                   & Z               
		\end{tikzcd} \]
		there exists a 1-cell $ u_\mu : W \rightarrow (f \downarrow g)$, unique up to a unique invertible 2-cell, and a pair of invertible 2-cells $ \alpha_s, \alpha_t$ such that: 
		\[ \begin{tikzcd}
		W \arrow[r, "t"] \arrow[d, "s"']  & Y \arrow[d, "g", ""{name=U, inner sep=0.1pt, below}] \\
		X \arrow[r, "f"', ""{name=D, inner sep=0.1pt}]          \arrow[from= D, to=U, Rightarrow, bend left=20]{}{\mu}                                                   & Z               
		\end{tikzcd} =  \begin{tikzcd}
		W \arrow[rrd, "t", bend left=20, ""{name=Ut, inner sep=0.1pt, below}] \arrow[rdd, "s"', bend right=20, ""{name=Us, inner sep=0.1pt}] \arrow[rd, "{u_\mu}" description, dashed, ""{name= Dt}, ""{name=Ds, below}] &                                                                                              &                  \\ \arrow[from = Ut, to=Dt, phantom]{}{\alpha_t \atop \simeq}
		\arrow[from = Us, to=Ds, phantom]{}{\alpha_s \atop \simeq}
		& X \times^{f,g}_Z Y \arrow[d, "f^*g"'] \arrow[r, "g^*f"]  & Y \arrow[d, "g", ""{name=U, inner sep=0.1pt, below}] \\
		& X \arrow[r, "f"', ""{name=D, inner sep=0.1pt}]             \arrow[from= D, to=U, Rightarrow, bend left=20]{}{\lambda_{f,g}}                                                                     & Z 
		\end{tikzcd}                                \]
		In other words, we have an equivalence in $\textup{Cat}$
		\[ \mathcal{K}[ W, (f \downarrow g)] \simeq  \mathcal{K}[W,f] \downarrow \mathcal{K}[W,g].\]
	\end{itemize}
\end{definition}

\begin{definition}
	Let $\mathcal{K}$ be a bicategory and $X$ a 0-cell. Then a \emph{bipower with 2 of $X$} is a 2-cell
	\[ 
	\begin{tikzcd}
	X^2 \arrow[rr, "\partial_0", bend left, ""{name=U, inner sep=0.1pt, below}] \arrow[rr, "\partial_1"', bend right, ""{name=D, inner sep=0.1pt}] &  & X \arrow[from= U, to=D, Rightarrow]{}{\lambda_X}  
	\end{tikzcd} \]
	where $ \partial_0$ and $ \partial_1$ are called respectively \emph{the universal domain} and \emph{universal codomain} of $X$, such that for any 2-cell 
	\[ \begin{tikzcd}
	Y \arrow[rr, "f", bend left, ""{name=U, inner sep=0.1pt, below}] \arrow[rr, "g"', bend right, ""{name=D, inner sep=0.1pt}] &  & X \arrow[from= U, to=D, Rightarrow]{}{\mu}  
	\end{tikzcd} \]
	there exists a 1-cell $ u_\mu : W \rightarrow X^2$, unique up to a unique invertible 2-cell, and a pair of invertible 2-cells $\alpha_f, \alpha_g$ such that:
	\[ \begin{tikzcd}
	Y \arrow[rr, "f", bend left, ""{name=U, inner sep=0.1pt, below}] \arrow[rr, "g"', bend right, ""{name=D, inner sep=0.1pt}] &  & X \arrow[from= U, to=D, Rightarrow]{}{\mu}  
	\end{tikzcd} = 
	\begin{tikzcd}
	Y \arrow[rrr, "f", bend left=49,""{name=Uf, inner sep=0.1pt, below}] \arrow[rrr, "g"', bend right=49, ""{name=Ug, inner sep=0.1pt}] \arrow[r, "u_\mu", dashed, ""{name=Df, inner sep=0.1pt, below}, ""{name=Dg, inner sep=0.1pt}] & X^2 \arrow[rr, "\partial_0", bend left, ""{name=U, inner sep=0.1pt, below}] \arrow[rr, "\partial_1"', bend right, ""{name=D, inner sep=0.1pt}] &  & X  \arrow[from = Uf, to=Df, phantom]{}{\alpha_f \atop \simeq} \arrow[from = Ug, to=Dg, phantom]{}{\alpha_g \atop \simeq} \arrow[from= U, to=D, Rightarrow]{}{\lambda_X}
	\end{tikzcd} \]
	In other words, we have an equivalence of categories 
	\[ \textup{Cat}[2, \mathcal{K}[Y,X]] \simeq \mathcal{K}[Y, X^2]. \]
\end{definition}

\begin{remark}
	Bipullbacks, bicommas and bipowers are related as follows: the bipower of $ X$ is the bicomma of the identity 1-cell with itself
	\[ \begin{tikzcd}
	X^2 \arrow[r, "\partial_0"] \arrow[d, " \partial_1 "'] & X \arrow[equal, d, "", ""{name=U, inner sep=0.1pt, below}] \\
	X \arrow[equal, r, ""', ""{name=D, inner sep=0.1pt}]  \arrow[from= D, to=U, Rightarrow, bend left=20]{}{\lambda_{X}}                               & X     \end{tikzcd} \]
	and the bicomma of two arrows $ f,g$ can be constructed from bipowers and pullbacks as:
	\[ 
	\begin{tikzcd}
	(f\downarrow g) \arrow[r] \arrow[d] \arrow[rd, "\lrcorner", near start, phantom] & {Z \times^{g,\partial_0}_X X^2} \arrow[d] \arrow[r] \arrow[rd, "\lrcorner", near start, phantom] & Z \arrow[d] \\
	{Y \times^{f,\partial_1}_X X^2} \arrow[d] \arrow[r] \arrow[rd, "\lrcorner", near start, phantom] & {X^2}        \arrow[r, "\partial_0"] \arrow[d, " \partial_1 "']                                               & X \arrow[equal, d, "", ""{name=U, inner sep=0.1pt, below}]          \\
	Y \arrow[r]                                                        & X \arrow[equal, r, ""', ""{name=D, inner sep=0.1pt}]  \arrow[from= D, to=U, Rightarrow, bend left=20]{}{\lambda_{X}}                                                                &           X 
	\end{tikzcd} \]
\end{remark}
\begin{remark}
	Whenever they exist, bilimits are unique up to unique equivalence in $\mathcal{K}$. 
\end{remark}

There are other shapes of finite bilimits that can be defined, and that exist in the bicategory of Grothendieck toposes, but we only need the above ones in this paper. The following fact is standard, and its proof can be found either in \cite{elephant} or in  \cite{Zoghaib}:

\begin{proposition}
	The bicategory of Grothendieck toposes is finitely bicomplete; in particular it has bipullbacks, bicommas and bipowers with $2$. 
\end{proposition}

\begin{proposition}\label{overtopos at the universal domain}
	The universal codomain $ \partial_1 : \Set[\mathbb{T}]^2 \rightarrow \Set[\mathbb{T}]$ is the over-topos of $ \Set[\mathbb{T}]$ at the universal domain $ U_\mathbb{T}$, that is, we have a geometric equivalence 
	\[ \Set[\mathbb{T}]^2 \simeq \Set[\mathbb{T}][U_\mathbb{T}] \]
	and an invertible 2-cell:
	\[ 
	\begin{tikzcd}
	{\Set[\mathbb{T}]^2} \arrow[rd, "\partial_1"'] \arrow[rr, "\simeq"] &  \arrow[d, "\simeq", phantom] & {\Set[\mathbb{T}][U_\mathbb{T}]} \arrow[ld, "u_{U_\mathbb{T}}"] \\
	& {\Set[\mathbb{T}]}           &                                                                       
	\end{tikzcd} \]
\end{proposition}

\begin{proof}
	This is a consequence of the universal property of the $ \mathbb{T}$-over-topos and the way in which a $\mathbb T$-model $M$ in $\cal E$ corresponds to a geometric morphism $g_{M}:{\cal E}\to \Set[{\mathbb T}]$ by the universal property of classifying toposes. Indeed, any object of $ \textup{\bf Geom}_\mathcal{E}[ g_{M}, \partial_1]$ is an invertible 2-cell 
	\[ 
	\begin{tikzcd}
	\mathcal{E} \arrow[rr, "f"] \arrow[rd, "g_{M}"'] &         \arrow[d, phantom, "\simeq" description]                  & {\Set[\mathbb{T}]^2} \arrow[ld, "\partial_1"] \\
	& {\Set[\mathbb{T}]} &                                                     
	\end{tikzcd} \]
	which produces a $2$-cell:
	\[ \begin{tikzcd}[column sep=large]
	\mathcal{E} \arrow[rr, "\partial_0^*f", bend left=25, ""{name=U, inner sep=0.1pt, below}, end anchor=160] \arrow[rr, "\partial_1^*f \simeq g_{M}"', bend right=25, ""{name=D, inner sep=0.1pt}, end anchor=200] &  & \Set[\mathbb{T}] \arrow[from= U, to=D, Rightarrow]{}{\lambda_{\Set[\mathbb{T}]}^*f}  
	\end{tikzcd} \]
	But this is exactly the same as a morphism of $\mathbb{T}$-models  $f : N \rightarrow M $ in $\mathcal{E}$. 
\end{proof}

\begin{corollary} For any geometric theory $\mathbb T$, $(\mathcal{C}_\mathbb{T}^2, J_{U_\mathbb{T}})$ is a small, cartesian, subcanonical site for $ \Set[\mathbb{T}]^2$. 
\end{corollary}

\begin{proof}
	It suffices to apply the construction of the $\mathbb T$-over-topos provided by Theorem \ref{thmovertopos} in the particular case ${\cal E}=\Sh({\cal C}_{\mathbb T}, J_{\mathbb T})$, $M=U_{\mathbb T}$ and $({\cal C}, J)=({\cal C}_{\mathbb T}, J_{\mathbb T})$. Note that $i_{\cal C}$ is simply the Yoneda embedding $y_{\mathbb T}:{\cal C}_{\mathbb T}\hookrightarrow \Sh({\cal C}_{\mathbb T}, J_{\mathbb T})$, and also $F_{M}$ coincides with $y_{\mathbb T}$, whence the site for $ \Set[\mathbb{T}][U_\mathbb{T}]$ provided by Definition \ref{overtopos} is $(\y_{\mathbb T} \downarrow \y_{\mathbb T}) \simeq \mathcal{C}_\mathbb{T}^2$, that is, the arrow category of $\mathcal{C}_\mathbb{T}$.
\end{proof}

\begin{proposition}\label{propovertoposcomma}
	Let $ \mathcal{E}$ be a Grothendieck topos and $ M$ in $ \mathbb{T}[\mathcal{E}]$, with $ \mathcal{E}[M]$ the associated $\mathbb{T}$-over-topos. Then the following square
	\[ \begin{tikzcd}
	\mathcal{E}[M] \arrow[]{r}{}  \arrow[]{d}[swap]{u_{M}} & \Set[\mathbb{T}]^2 \arrow[]{d}{\partial_1} \\ \mathcal{E} \arrow[]{r}{g_{M}} & \Set[\mathbb{T}]
	\end{tikzcd}\] 
	is a bipullback, where the top morphism is the name of the universal geometric transformation from $g_{M}\circ u_{M}$ to the canonical morphism ${\cal E}[M]\to \Set[{\mathbb T}]$.
\end{proposition}

\begin{proof}
	This is just unravelling of universal properties. Suppose we have a square: 
	\[ 
	\begin{tikzcd}
	\mathcal{G} \arrow[d, "g"'] \arrow[r, "f"] & {\Set[\mathbb{T}]^2} \arrow[d, "\partial_1"] \\
	\mathcal{E} \arrow[r, "g_{M}"]                 & {\Set[\mathbb{T}]}                          
	\end{tikzcd} \]
	This yields a morphism of $ \mathbb{T}$-models $f $ in $\mathbb{T}[G]$ such that $ \partial_1 f = g_{M}\circ g$, that is, an object of $ \mathbb{T}[\mathcal{G}]/g^*M$ defining a unique morphism $\overline{f}$ as in the following diagram: 
	\[ 
	\begin{tikzcd}
	\mathcal{G} \arrow[rd, "\overline{f}", dashed] \arrow[rdd, "g"', bend right=30] \arrow[rrd, "f", bend left=20] &                                                     &                                                     \\
	& {\mathcal{E}[M]} \arrow[d, "u_M"'] \arrow[r, "1_M"] & {\Set[\mathbb{T}]^2} \arrow[d, "\partial_1"] \\
	& \mathcal{E} \arrow[r, "g_{M}"']                         & {\Set[\mathbb{T}]}                          
	\end{tikzcd} \]
	Hence by the uniqueness up to unique equivalence of the bipullback, the $\mathbb{T}$-over-topos coincides with the bipullback.
\end{proof}

\begin{remark}
	Equivalently, we can define $\mathcal{E}[M]\to {\cal E}$ as a comma topos:
	\[ \begin{tikzcd}
	\mathcal{E}[M] \arrow[]{d} \arrow[]{r}{u_{M}} & \mathcal{E} \arrow[""{name=U, below}]{d} \\ 
	\Set[{\mathbb T}] \arrow[equal, ""{name=D}]{r}{} &   \Set[\mathbb{T}] \arrow[from=D, to=U, Rightarrow, bend left=20]{}{\lambda_M}
	\end{tikzcd}\]
\end{remark}

\begin{proposition}
	Let be $ M$ be in $ \mathbb{T}[\mathcal{E}]$ and $ g : \mathcal{G} \rightarrow \mathcal{E}$ be a geometric morphism. Then the following square is a pullback:
	\[ 
	\begin{tikzcd}
	{\mathcal{G}[g^*{M}]} \arrow[r] \arrow[d] \arrow[rd, "\lrcorner", very near start, phantom] & {\mathcal{E}[M]} \arrow[d] \\
	\mathcal{G} \arrow[r]                                                             & \mathcal{E}               
	\end{tikzcd} \]
\end{proposition}

\begin{proof}
	This is a simple application of the bipullback lemma: in the diagram below
	\[ 
	\begin{tikzcd}
	{\mathcal{G}[g^*\mathbb{M}]} \arrow[r] \arrow[d, "u_{g^*M}"'] & {\mathcal{E}[M]} \arrow[d, "u_{M}"] \arrow[r] & {\Set[\mathbb{T}][U_\mathbb{T}]} \arrow[d, "u_{U_\mathbb{T}}"] \\
	\mathcal{G} \arrow[r, "g"]                                    & \mathcal{E} \arrow[r, "M"]                    & {\Set[\mathbb{T}]}                                            
	\end{tikzcd} \]
	both the right and outer squares are bipullbacks, hence the left one is so.
\end{proof}

We now turn to the geometric interpretation of the over-topos construction.

\begin{definition}[Theorem C3.6.16 \cite{elephant}]
	A geometric morphism $f$ is said to be \emph{totally connected} if the following equivalent conditions are fulfilled:\begin{itemize}
		\item $f^*$ has cartesian left adjoint $ f_!$;
		\item $ f$ has a right adjoint in the 2-category of Grothendieck toposes;
		\item $ f$ has a terminal section.
\end{itemize}\end{definition}
Morally, a topos is totally connected if it has a terminal point, dually to local toposes who have an initial point. In particular, the universal codomain morphism $ \partial_1$ of a topos is totally connected (as the universal domain morphism $ \partial_0$ is local). Totally connected geometric morphisms are stable under pullback (cf. Lemma C3.6.18 \cite{elephant}), so, by Proposition \ref{propovertoposcomma}, the $\mathbb{T}$-over-topos $ \mathcal{E}[M] \rightarrow \mathcal{E}$ is totally connected; indeed, its terminal point is the identity morphism $1_M$. In particular, all the fibers of this morphisms are also totally connected, while its sections 
\[ \begin{tikzcd} 
& \mathcal{E}[M] \arrow[]{d}{u_{M}} \\ 
\mathcal{E} \arrow[]{ur}{s} \arrow[equal]{r}{} & \mathcal{E}
\end{tikzcd} \]
just are the name of some homomorphism $N \rightarrow M$ in $\mathcal{E}$. \\

Let us now analyse the $\mathbb T$-over-topos construction in the particular case $ \mathbb{T} = \mathbb{O}$, the (one-sorted) theory of objects. This yields a \ac thickening' of the notion of slice topos; indeed, for any object $E$ in a topos $\mathcal{E}$, the ``totally connected component" of $\mathcal{E}$ in $E$ 
\[ \begin{tikzcd}
\mathcal{E}[E] \arrow[]{r}{}  \arrow[phantom, very near start]{rd}{\lrcorner} \arrow[]{d}[swap]{u_{E}} & \Set[\mathbb{O}]^2 \arrow[]{d}{\partial_1} \\ \mathcal{E} \arrow[]{r}{E} & \Set[\mathbb{O}]
\end{tikzcd}\] 
classifies generalized elements $ F \rightarrow E$ of the object $E$, as opposed to the usual étale topos
\[ \begin{tikzcd}
\mathcal{E}/E  \arrow[phantom, very near start]{rd}{\lrcorner} \arrow[]{r}{} \arrow[]{d}[swap]{\widehat{E}} & \Set[\mathbb{O}_\bullet] \arrow[]{d}{} \\ \mathcal{E} \arrow[]{r}{E} & \Set[\mathbb{O}]
\end{tikzcd}\] 
which is constructed from the universal fibration of the classifier of pointed objects over the classifier of objects. Indeed, whilst the \emph{objects} of $\widehat{E}:{\cal E}\slash E \to {\cal E}$ are generalized elements of $E$, that is, morphisms $F \rightarrow E$, $\widehat{E}$ actually classifies over $\cal E$ \emph{global} elements of $E$ in the sense that its sections 
\[ \begin{tikzcd} 
& \mathcal{E}/E \arrow[]{d}{\widehat{E}} \\ 
\mathcal{E} \arrow[]{ur}{s} \arrow[equal]{r}{} & \mathcal{E}
\end{tikzcd} \]
correspond to global elements $ 1 \rightarrow E$. In particular, its fibers at points are discrete (as $ E$ only has a \emph{set} of global elements). 

Note that any point of $\mathcal{E}[E]$ defines by composition with $u_{E}$ a point $ p $ of $\mathcal{E}$ over which it lies. By the universal property of the pullback, it thus yields a section $s$ of the totally connected geometric morphism at the corresponding fiber:
\[
\begin{tikzcd}[sep=large]
& {\Set[p^*E]} \arrow[d, "u_{p^{\ast}(E)}"] \arrow[r] \arrow[rd, "\lrcorner", phantom, very near start] & {\mathcal{E}[E]} \arrow[d, "u_{E}"] \\
\Set \arrow[r, equal] \arrow[ru, "s"] & \Set \arrow[r, "p"']                                                                           & \mathcal{E}                               
\end{tikzcd}\]
Note that $s$ defines a generalized element $X \rightarrow  p^{\ast}(E)$ of the stalk of $E$ at $p$, equivalently a $X$-indexed family of global elements of $p^{\ast}(E)$; this should be compared with the \'{e}tale case, where points of the \'{e}tale topos are mere global elements of the stalk. 

Finally, we observe that the two constructions are related as in the following diagram, where the upper square is a pullback as the bottom and front square are:
\[ 
\begin{tikzcd}
\mathcal{E}/E  \arrow[dd, "\widehat{E}"'] \arrow[rd] \arrow[rr] \arrow[rrrd, "\lrcorner", phantom, very near start] &                                                                                                     & {\Set[\mathbb{O}_\bullet] } \arrow[rd] \arrow[dd] &                                                      \\
& {\mathcal{E}[E]} \arrow[ld, "u_{E}" description] \arrow[r, no head] \arrow[rd, "\lrcorner", phantom, very near start] &                        \arrow[r]                                  & {\Set[\mathbb{O}]^2} \arrow[ld, "\partial_1"] \\
\mathcal{E}  \arrow[rr, "E"']                                                                      &                                                                                                     & {\Set[\mathbb{O}]}                                &                                                     
\end{tikzcd}\]

\vspace{0.4cm}

\textbf{Acknowledgements:} We thank Laurent Lafforgue for his careful reading of a preliminary version of this paper.

\vspace{1cm}

\textsc{Olivia Caramello} 

\vspace{0.2cm}
{\small \textsc{Dipartimento di Scienza e Alta Tecnologia, Universit\`a degli Studi dell'Insubria, via Valleggio 11, 22100 Como, Italy.}\\
	\emph{E-mail address:} \texttt{olivia.caramello@uninsubria.it}}

\vspace{0.2cm}

{\small \textsc{Institut des Hautes \'Etudes Scientifiques, 35 Route de Chartres,
		91440 Bures-sur-Yvette, France.}\\
	\emph{E-mail address:} \texttt{olivia@ihes.fr}}

\vspace{0.6cm}

\textsc{Axel Osmond} 

\vspace{0.2cm}
{\small \textsc{IRIF, 8 Place Aurélie Nemours, 75013 Paris, France.}\\
	\emph{E-mail address:} \texttt{Axel.Osmond@irif.fr}}


\begin{thebibliography}{10}
	\bibitem{grothendieck} M. Artin, A. Grothendieck and J. L. Verdier, \emph{Th\'eorie des topos et cohomologie \'etale des sch\'emas} - (SGA 4), S\'eminaire de G\'eom\'etrie Alg\'ebrique du Bois-Marie, ann\'ee 1963-64; second edition published as Lecture Notes in Math., vols 269, 270 and 305 (Springer-Verlag, 1972). 
	
	\bibitem{OCBook} O. Caramello, \emph{Theories, Sites, Toposes: Relating and studying mathematical theories through topos-theoretic `bridges'} (Oxford University Press, 2017).  
	
    \bibitem{OCdenseness} O. Caramello, Denseness conditions, morphisms and equivalences of toposes,  \emph{arxiv:math.CT/1906.08737v3} (2020). 
	
	\bibitem{elephant} P. T. Johnstone, \emph{Sketches of an Elephant: a topos theory compendium. Vols. 1 and 2}, vols. 43 and 44 of \emph{Oxford Logic Guides} (Oxford University Press, 2002).
	
	\bibitem{giraud} J. Giraud, Classifying topos, in \emph{Toposes, Algebraic Geometry and Logic}, pp. 43-56 (Springer, 1972).
	
	\bibitem{Zoghaib} S. Zoghaib, Théorie homotopique de Morel-Voevodsky et applications, 2020, available at \texttt{http://www.normalesup.org/~zoghaib/math/dea.pdf}.
\end{thebibliography}
\end{document}